\documentclass[11pt]{amsart}
\usepackage{fullpage, amsmath, amssymb}
\title{Computational approaches to Poisson traces associated to finite subgroups of $\Sp_{2n}(\CC)$}
\author{Pavel Etingof, Sherry Gong, Aldo Pacchiano, Qingchun Ren, and Travis Schedler}
\usepackage{url}            
\usepackage{graphicx}       

\numberwithin{equation}{section}

\theoremstyle{definition}
\newtheorem{theorem}[equation]{Theorem}
\newtheorem{lemma}[equation]{Lemma}
\newtheorem{proposition}[equation]{Proposition}
\newtheorem{corollary}[equation]{Corollary}

\newtheorem{definition}[equation]{Definition}

\newtheorem{question}[equation]{Question}

\newtheorem{remark}[equation]{Remark}

\newtheorem{claim}[equation]{Claim}

\newcommand{\Id}{\operatorname{Id}}
\newcommand{\rk}{\operatorname{rk}}
\newcommand{\ev}{\mathrm{ev}}
\newcommand{\Span}{\operatorname{Span}}
\newcommand{\ad}{{\operatorname{ad}}}

\newcommand{\HH}{\mathsf{HH}}

\newcommand{\HP}{\mathsf{HP}}
\newcommand{\RR}{\mathbb{R}}

\newcommand{\ZZ}{\mathbb{Z}}

\newcommand{\iso}{{\;\stackrel{_\sim}{\to}\;}}
\newcommand{\C}{\mathbb{C}}

\newcommand{\cO}{\mathcal{O}}

\newcommand{\Hom}{\text{Hom}}

\newcommand{\SL}{\mathsf{SL}}
\newcommand{\GL}{\mathsf{GL}}
\newcommand{\Sp}{\mathsf{Sp}}
\newcommand{\gr}{\operatorname{\mathsf{gr}}}

\newcommand{\onto}{\twoheadrightarrow}
\newcommand{\into}{\hookrightarrow}
\newcommand{\Sym}{\operatorname{\mathsf{Sym}}}
\newcommand{\Mat}{\mathrm{Mat}}
\newcommand{\Diag}{\mathrm{Diag}}

\newcommand{\cD}{\mathcal D}

\newcommand{\lcm}{\operatorname{lcm}}

\newcommand{\CC}{{\Bbb C}}
\newcommand{\bbA}{\mathbb{A}}
\newcommand{\FF}{{\Bbb F}}
\begin{document}
\date{Jan 26, 2011}
\begin{abstract} We reduce the computation of Poisson traces on
  quotients of symplectic vector spaces by finite subgroups of
  symplectic automorphisms to a finite one, by proving several results
  which bound the degrees of such traces as well as the dimension in
  each degree.  This applies more generally to traces on all
  polynomial functions which are invariant under invariant Hamiltonian
  flow.  We implement these approaches by computer together
  with direct computation for infinite families of groups, focusing on
  complex reflection and abelian subgroups of $\GL_2(\CC) <
  \Sp_4(\CC)$, Coxeter groups of rank $\leq 3$ and $A_4, B_4=C_4$, and
  $D_4$, and subgroups of $\SL_2(\CC)$.
  
\end{abstract}
\maketitle
\tableofcontents

\section{Introduction}
Let $A$ be a Poisson algebra over $\C$.  We are interested in linear
functionals $A \rightarrow \C$ satisfying $\{a, b\} \mapsto 0$ for all
$a,b \in A$. Such functionals are called \emph{Poisson traces} on $A$.
The space of Poisson traces is denoted by $\HP_0(A)^*$, and is dual to
the vector space $\HP_0(A) := A / \{A, A\}$, known as the \emph{zeroth
  Poisson homology}, which coincides with the zeroth Lie homology.

Here, we study the case where $A = \cO_V^G$ is the algebra of
$G$-invariant polynomial functions on a nonzero symplectic vector
space $V$, for a finite subgroup $G < \Sp(V)$. We will let $2n > 0$
denote the dimension of $V$. We also consider the
larger space $\HP_0(\cO_V^G, \cO_V) := \cO_V/\{\cO_V^G, \cO_V\}$, as
well as
its dual, $\HP_0(\cO_V^G, \cO_V)^*$, which is the space of functionals
$\phi$ on $\cO_V$ which are invariant under the flow of $G$-invariant
Hamiltonian vector fields, i.e., $\phi(\{f, g\})=0$ for all $f \in
\cO_V^G$ and $g \in \cO_V$.  Note that  $\HP_0(\cO_V^G,
\cO_V)^*$ is a $G$-representation, and its $G$-invariants form the
space of Poisson traces on $\cO_V^G$.

In general, not very much is known about such Poisson traces.  In
\cite{AFLS}, a related quantity was computed: the dimension of the
space of Hochschild traces on $\cD_X^G$ where $\cD_X$ is the algebra
of differential operators on $X \subseteq V$, a Lagrangian subspace.
The algebra $\cD_X^G$ is naturally a quantization of $\cO_V^G$, and
its Hochschild traces are defined as $\HH_0(\cD_X^G)^* := (\cD_X^G /
[\cD_X^G, \cD_X^G])^*$.  More precisely, equip $\cO_V$ with its
natural grading by degree of polynomials and $\cD_X$ with its natural
filtration (which is known as the additive or Bernstein
filtration). Then, $\gr \cD_X = \cO_V$, and there is a canonical
surjection $\HP_0(\cO_V^G) \onto \gr \HH_0(\cD_X^G)$, and similarly
$\HP_0(\cO_V^G, \cO_V) \onto \gr \HH_0(\cD_X^G, \cD_X)$. As a result,
the dimension of the space of Hochschild traces is a lower bound for
the dimension of the space of Poisson traces.  In some special cases,
the lower bound is attained, i.e., the surjection is an isomorphism.
For example, $\HP_0(\cO_V^G) \cong \gr \HH_0(\cD_X^G)$ is known to
hold when $V = \C^2$, and in \cite{ESsym}, the first and last authors
generalized this to the case $V = \C^{2n} = (\C^2)^{\oplus n}$ and $G
= S_n \ltimes K^n$ for $K < \SL_2(\CC)$ (certain cases were shown
previously in \cite{Bu}, and this result was conjectured by Alev
\cite[Remark 40]{Bu}). In \cite{ESweyl}, the same authors will show
that $\HP_0(\cO_V^G) \cong \gr \HH_0(\cD_X^G)$ when $G = S_{n+1}$ is a
Weyl group of type $A_n$ acting on its reflection representation $V =
\CC^{2n}$ (but not for the $D_n$ case).

The following explicit formula for $\HH_0(\cD_X^G, \cD_X)$ as a
$G$-representation is an easy generalization of the main result of
\cite{AFLS}.  Let $\C[G]_\ad$ denote the $G$-representation with
underlying vector space the group algebra $\C[G]$, but with the
conjugation action of $G$.  
\begin{lemma}  \label{l:afls-fla}
  As a $G$-representation, $\HH_0(\cD_X^G, \cD_X)$ is
  isomorphic to the subrepresentation of $\C[G]_\ad$ spanned by
  elements $g \in G$ such that $g - \Id$ is invertible. 
\end{lemma}
We stress, however, that the above lemma does \emph{not} say anything
about the filtration on $\HH_0(\cD_X^G, \cD_X)$ and hence about the
grading on $\gr \HH_0(\cD_X^G, \cD_X)$.  In the aforementioned cases in
\cite{ESsym} and \cite{ESweyl}, $\HP_0(\cO_X^G)$ is computed along with
its grading, so when it is also isomorphic to $\gr \HH_0(\cD_X^G)$,
one obtains the grading on the latter.

Although we will not use it, the argument of Lemma \ref{l:afls-fla}
applies more generally to show that $\HH_*(\cD_X^G, \cD_X) \cong
\C[G]_\ad$ as $G$-representations, with $\HH_j(\cD_X^G, \cD_X)$ mapping 
to the span of elements $g$ such that $\rk(g - \Id) =
\dim V - j$.  In particular, $\HH_*(\cD_X^G, \cD_X)$ is always
finite-dimensional. This is not necessarily true for $\HP_*(\cO_V^G,
\cO_V)$: see, e.g., \cite[Theorem 2.4.1.(ii)]{EGdelpezzo}, which
implies that $\HP_*(\cO_V^G)$ is infinite-dimensional when $G$ is
nontrivial and $V$ is two-dimensional.

However, thanks to \cite[\S 7]{BEG} (see also \cite{ESdm}), the space
$\HP_0(\cO_V^G, \cO_V)$ is finite-dimensional. On the other hand,
explicit upper bounds are known in only a few cases.  The first aim of
this paper is to prove explicit upper bounds, which allow us to
compute precisely $\HP_0(\cO_V^G, \cO_V)$ and $\HP_0(\cO_V^G)$ for
small enough $G$ and low enough dimension of $V$ with the help of
computer programs.

More precisely, it is not very computationally useful to prove an
upper bound on $\dim \HP_0(\cO_V^G, \cO_V)$, since this does not
immediately render its computation finite.  Instead, we find upper
bounds on the top degree of $\HP_0(\cO_V^G, \cO_V)$ as a graded vector
space.  This renders the computation of $\HP_0(\cO_V^G, \cO_V)$
finite.

To prove such a bound, we use the following reformulation exploited in
\cite[\S 7]{BEG}.  Given any Poisson algebra $A$ and any $f \in A$,
the condition that a functional $\varphi \in A^*$ kills $\{f, A\}$ can
be rewritten as $\xi_f(\varphi) = 0$ where $\xi_f$ is the Hamiltonian
vector field corresponding to $f$, which acts on $A$ by
$\xi_f(g)=\{f,g\}$ and acts on $A^*$ by the negative dual.  In the
case that $A = \cO_V$ is a polynomial algebra, we may canonically
identify the \emph{graded} dual $A^*$, defined by $(A^*)_{i} :=
(A_{-i})^*$, with $\cO_{V^*}$.
Call this isomorphism $F:
A^* \iso \cO_{V^*}$.  Under this isomorphism, 
\begin{equation}
F(\xi_f(\varphi)) = F_D(\xi_f) F(\varphi),
\end{equation}
where $F_D(\xi_f)$ is a kind of Fourier transform of $\xi_f$: for
every $v \in V^*, w \in V$, and $m \geq 0$, $F_D(v^m \partial_w) =
w \partial_v^m$. Here, $\partial_v, \partial_w$ are differentiation
operators defined by $\partial_w(v) = v(w) = \partial_v(w)$.  More
generally, $F_D: \cD_V \iso \cD_{V^*}^{\mathrm{op}}$ is an
anti-isomorphism of rings of differential operators, given by $v
\mapsto \partial_v$ and $\partial_{w} \mapsto w$.

As a result, $\HP_0(\cO_{V}^G, \cO_V)^*$ is identified with the
solutions $h \in \cO_{V^*}$ of the differential equations
\begin{equation}
F_D(\xi_f) (h) = 0, \forall f \in \cO_V^G.
\end{equation}
To help understand the main argument below, we will make the above
explicit using coordinates (although we do not strictly need to do
this---everything below can be formulated invariantly. We will at
least take care to distinguish between vector spaces and their duals.)
Suppose that $\cO_V^G$ is generated as a commutative algebra by elements $h_1,
\ldots, h_k$, and $V = X \oplus Y$ is symplectic with complementary
Lagrangians $X$ and $Y$.  Let us write $V^* = X^* \oplus Y^*$, where
the inclusions $X^*, Y^* \subseteq V^*$ are defined by $X^* = Y^\perp$
and $Y^* = X^\perp$.  Fix bases $(x_1, \ldots, x_n)$ and $(y_1,
\ldots, y_n)$ of $X^*$ and $Y^*$, respectively, with dual bases
$(x_1^*, \ldots, x_n^*)$ and $(y_1^*, \ldots, y_n^*)$ of $X$ and $Y$,
and assume that $(x_i^*, y_j^*) = \delta_{ij} = -(y_j^*, x_i^*)$. In particular,
$\cO_V = \C[x_1, \ldots, x_n, y_1, \ldots, y_n]$.  This induces the
isomorphism $V \iso V^*$ given by $x_i \mapsto y_i^*$ and $y_i \mapsto
- x_i^*$, and hence the Poisson bracket $\{x_i, y_j\} = \delta_{ij} =
- \{y_j, x_i\}$.  Then, $\HP_0(\cO_V^G, \cO_V)^* \subseteq \cO_{V^*}$
identifies with the solutions of the differential equations
\begin{equation} \label{ljvgens} \sum_{i=1}^n \bigl( y_i^*
  F_D(\frac{\partial h_j}{\partial x_i}) - x_i^* F_D(\frac{\partial h_j}{\partial y_i})
  \bigr)(g) = 0.
\end{equation}
  Note that, in \eqref{ljvgens},
we only needed the restriction of $F_D$ to $\cO_V$,
\begin{equation} \label{e:FDrestr} F_D: \C[x_1, \ldots, x_n, y_1,
  \ldots, y_n] \iso \C[\partial_{x_1},
  \ldots, \partial_{x_n}, \partial_{y_1}, \ldots, \partial_{y_n}].
\end{equation}
The reason why we wrote $\frac{\partial h_j}{\partial x_i}$ instead of $\partial_{x_i^*} (h_j)$ above was to avoid confusion with the product of the two elements $\partial_{x_i^*}, h_j \in \mathcal{D}_{V^*}$, which would not be in $\cO_V$, and similarly with $\frac{\partial h_j}{\partial y_i}$.

Next, for every $v \in V^*$, we can evaluate the above equations at $v$:
\begin{equation} \label{jvgens} \sum_{i=1}^n \bigl(y_i^*(v)
  F_D(\frac{\partial h_j}{\partial x_i}) - x_i^*(v) F_D(\frac{\partial h_j}{\partial y_i})\bigr)
  (g)(v) = 0.
\end{equation}
This shows that the Taylor coefficients $F(\partial_{x_1},
\ldots, \partial_{x_n}, \partial_{y_1}, \ldots, \partial_{y_n})(g)(v)$
of $g$ at $v$ (for $F$ a polynomial) only depend on the class of $F$
in the quotient $R_v :=
\C[\partial_{x_1},\ldots,\partial_{x_n},\partial_{y_1},\ldots,\partial_{y_n}]
/ J_v$ (and on $g$), where $J_v$ is the ideal generated by the
constant-coefficient operators on the LHS of \eqref{jvgens}, i.e., the
elements $D_{v'}h_1, \ldots, D_{v'} h_k$ where $v' \in V$ is the
element corresponding to $v \in V^*$ via the symplectic form, and
$D_{v'}$ is the directional derivative operation $D_{v'}: \cO_V \to
\cO_V$. Note that $J_v$ does not actually depend on the choice of
generators $h_1, \ldots, h_k \in \cO_V^G$, since if we adjoin another
polynomial $h_{k+1} \in \cO_V^G$ to the list $h_1, \ldots, h_k$, the
new equation \eqref{jvgens} is already implied by the previous $k$
equations due to the Leibniz rule, $D_{v'}(fg) = (D_{v'} f) g +
(D_{v'} g) f$.

As a result, we deduce that
\begin{equation}
\dim \HP_0(\cO_V^G, \cO_V)^* \leq \dim R_v, \forall v \in V^*.
\end{equation}
This is the upper bound found in \cite[Proposition 3.5]{ESdm} (with
the Fourier transform of the proof found there), and gives a precise
version of the proof that $\HP_0$ is finite-dimensional from \cite[\S
7]{BEG}, once one notices that $\dim R_v$ is finite for generic $v \in
V^*$.\footnote{This is true since the support of $J_v$ is generically
  $\{0\}$. This holds with minimal $\dim R_v$ when $v$ does not
  annihilate any subspace of the form $V^K$ for $K =
  \operatorname{Stab}_G(u) \neq \{1\}$ and $u \in V$: see
  \cite[Theorem 4.13]{ESdm}; cf.~\cite[\S 7]{BEG}. For a more general
  result which implies the generic finite-dimensionality of $R_v$, see
  Remark \ref{r:regseq} below.}  However, the main drawback is that there is no relation, in general, between the grading on $\HP_0(\cO_V^G, \cO_V)$ and that on $R_v$.  The first main goal of this paper is to overcome this problem.

Much of this paper will concern the special case where $G < \GL(X) <
\Sp(V)$, where the embedding $\GL(X)
< \Sp(V)$ is defined by sending $A \in \GL(X)$ to $A \oplus (A^{-1})^*
\in \Sp(X \oplus Y)$.

We now outline the contents of the paper.  First, \S \ref{s:koszul}
gives an elementary bound on $\dim R_v$ using regular sequences, using
an argument we will need again in \S \ref{glnsec}. We also apply these
results in \S \ref{ss:ncalg} to bound the number of irreducible
finite-dimensional representations of filtered quantizations as well
as the number of zero-dimensional symplectic leaves of filtered
Poisson deformations, although this is not needed for the rest of the
paper.

In \S$\!$\S \ref{glnsec} and \ref{matsec} we refine the argument
outlined in the present section in two different ways to obtain
computationally useful bounds on the top degree of $\HP_0(\cO_V^G,
\cO_V)$.  In \S \ref{glnsec}, we apply the above argument in the
case $v \in X^*$ and $G < \GL(X) < \Sp(V)$ to obtain an upper bound on
the top degree of $\HP_0(\cO_V^G, \cO_V)$.  In \S \ref{matsec},
for arbitrary $G$ (not necessarily preserving a Lagrangian subspace)
and for arbitrary $v \in V$ such that $R_v$ is finite-dimensional, we
define a square matrix $A_v$ of size $\dim R_v$ such that the
dimension of the degree $m$ part $\dim \HP_0(\cO_V^G, \cO_V)^*_m$ is
bounded by the dimension of the $m$-eigenspace of $A$. We do this by
lifting generators $f_1, \ldots, f_N$ of $R_v$ to differential
operators $F_1, \ldots, F_N$ on $V^*$, and considering the
differential equations satisfied by the vector $(F_1(T), \ldots,
F_N(T))$ for all $T \in \HP_0(\cO_V^G, \cO_V)^*$ upon evaluation on
the line $\C \cdot v$.

Next, in \S \ref{compsec}, we will apply these results and
computer programs \cite{ReSc-progs} written by two of the authors in
Magma \cite{magma} to obtain $\HP_0(\cO_V^G, \cO_V)$ for many groups
$G$, including all finite subgroups of $\SL_2(\CC)$, the Coxeter
groups of rank $\leq 3$ and types $A_4, B_4 = C_4$, and $D_4$, and the
exceptional Shephard-Todd complex reflection groups $G_4, \ldots,
G_{22} < \GL_2 < \Sp_4$ (except for $G_{18}$ and $G_{19}$, where we
could only obtain $\HP_0(\cO_V^G)$ and without proof).  Combining the
latter with results of \S \ref{s:gmp2}, we obtain a classification of
complex reflection groups of rank two for which $\HP_0(\cO_V^G, \cO_V)
\cong \gr \HH_0(\cD_X^G, \cD_X)$ as well as those for which
$\HP_0(\cO_V^G) \cong \gr \HH_0(\cD_X^G)$, and give the Hilbert series
in these cases.

In the final two sections, we explicitly compute $\HP_0(\cO_V^G,
\cO_V)$, as well as its grading and $G$-structure, for several
infinite families of groups in $\Sp_4$. Namely, in \S \ref{s:sp4abel},
we give an explicit description of $\HP_0(\cO_V^G, \cO_V)$ in the case
that $G < \Sp_4$ is abelian (where it coincides with
$\HP_0(\cO_V^G)$), classify such groups that have the property that
$\HP_0(\cO_V^G) \cong \gr \HH_0(\cD_X^G)$, and give the relevant
Hilbert series.  In \S \ref{s:gmp2}, we explicitly compute
$\HP_0(\cO_V^G, \cO_V)$ for the complex reflection groups
$G=G(m,p,2)$, and classify those having the properties $\HP_0(\cO_V^G,
\cO_V) \cong \gr \HH_0(\cD_X^G, \cD_X)$ and $\HP_0(\cO_V^G) \cong \gr
\HH_0(\cD_X^G)$.

Throughout this article, $G$ always denotes a \emph{finite} group, and
$V$ a \emph{finite-dimensional} symplectic vector space. The algebra
$\cO_V$ and the space $\HP_0(\cO_V^G, \cO_V)$ are nonnegatively
graded, whereas their duals, $\cO_{V^*}$ and $\HP_0(\cO_V^G,
\cO_V)^*$, are nonpositively graded.
\subsection{Acknowledgements}
This work grew out of several projects supported by
MIT's Undergraduate Research Opportunities Program.
The first author's work was partially supported by the NSF grant
DMS-1000113. The second author is a five-year fellow of the American
Institute of Mathematics, and was partially supported by the
ARRA-funded NSF grant DMS-0900233. 

\section{An elementary bound on dimension using Koszul complexes}\label{s:koszul}
We begin with an elementary explicit bound on the dimension of
$\HP_0(\cO_V^G, \cO_V)$. While, for computational purposes, we
ultimately want to bound its top degree,
 we include this both because it may be of
independent interest, and because we will generalize it in \S
\ref{glnkcsec} to give a bound also on the top degree. Additionally, in
the next subsection we apply it to representation theory.

We will consider $J_v$ to be an ideal of $\cO_V$ via
\eqref{e:FDrestr}.  If $h_1, \ldots, h_{2n} \in J_v$ is a collection
of homogeneous elements which forms a regular sequence, i.e., $h_i$ is
a nonzerodivisor in $\cO_V/(h_1, h_2, \ldots, h_{i-1})$ for all $i$,
then the Hilbert series of $R = \cO_V/(h_1, \ldots, h_{2n})$ can be
computed using the associated Koszul complex, and one obtains
\begin{equation}\label{rvkoseqn}
h(R_v;t) \leq h(R;t) = \frac{\prod_{i=1}^{2n} (1-t^{|h_i|})}{(1-t)^{2n}}.
\end{equation}
Here we say that $\sum_i a_i t^i \leq \sum_i b_i t^i$ if $a_i \leq
b_i$ for all $i$.

We can construct such a regular sequence from a regular sequence $g_1, \ldots, g_{2n} \in \cO_V^G$ using the following lemma, which essentially follows from \cite[Theorem 3.1]{ESdm}. We will actually state and prove it more
generally.
\begin{lemma} \label{l:regseq} Let $U$ be an arbitrary
  finite-dimensional vector space and $g_1, \ldots, g_{\dim U} \in
  \cO_U$ a regular sequence of homogeneous elements of degree $\geq
  2$.  Then, for generic $u \in U$, the directional derivatives $D_u
  g_1, \ldots, D_u g_{\dim U}$ also form a regular sequence.
\end{lemma}
\begin{remark} \label{r:regseq} In particular, the ideal in $\cO_U$
  generated by $D_u g_1, \ldots, D_u g_{\dim U}$ has finite
  codimension for generic $u$. Specializing to the case that $U=V$ is
  symplectic of dimension $2n > 0$, $G < \Sp(V)$ is finite, and $g_1,
  \ldots, g_{2n} \in \cO_V^G$, then for $v \in V^*$ and $u \in V$ the
  corresponding element by the symplectic form, this ideal is
  contained in $J_v$. Hence, this result strengthens the fact from
  \cite[\S 3]{ESdm} that $J_v$ has finite codimension for generic $v
  \in V^*$, once one notes that a
  regular sequence $g_1, \ldots, g_{2n} \in \cO_V^G$ of
  positively-graded homogeneous elements always exists (the elements
  must have degree $\geq 2$ unless $V^G \neq \{0\}$, in which case
  $J_v$ is generically the unit ideal).
\end{remark}
\begin{proof}
  We will prove that, for generic $u$, the vanishing locus $Y_u$
of the functions $D_u g_1, \ldots, D_u g_{\dim U}$ is $\{0\}$. Hence
they form a complete intersection, and therefore a regular sequence (by standard characterizations of regular
sequences; see, e.g., \cite[\S$\!$\S 17, 18]{Eis}).  Note that $Y_u$ is nonempty and invariant under scaling, since $g_1, \ldots, g_{\dim U}$ are homogeneous of degrees $\geq 2$. So we only need to prove that $\dim Y_u = 0$.

  The inclusion of polynomial algebras $\CC[g_1, \ldots, g_{\dim U}]
  \subseteq \cO_U$ defines a map $\phi: U \to \bbA^{\dim U}$.  Since $g_1,
  \ldots, g_{\dim U}$ define a regular sequence, $\phi$ is a finite
  map, i.e., $\cO_U$ is a finite module over the polynomial subalgebra
  $\CC[g_1, \ldots, g_{\dim U}]$.  Now, consider the locus
\[
Z := \{(v,u) \in T U \mid v \in U, u \in T_vU, D_ug_i(v) = 0, \forall i\}.
\]
We are interested in the intersection $(U \times \{u\}) \cap Z =
(Y_u \times \{u\})$.

For every $0 \leq r \leq \dim U$, consider the locus $U_r$ of $v \in
U$ at which the map $\phi$ has rank $r$, i.e., the derivatives
$D(g_1)|_{v}, \ldots, D(g_{\dim U})|_{v}$ evaluated at $v$ span a
dimension $r$ subspace of $T_v^* U$.  Then, the intersection $Z \cap
(TU|_{U_r})$ is a vector bundle of rank $\dim U-r$ over $U_r$.

We claim that $\dim U_r \leq r$. This implies that $\dim Z \leq \dim
U$.  Thus, $(U \times \{u\}) \cap Z = (Y_u \times \{u\})$ has
dimension zero for generic $u$ (as $Y_u$ is always nonempty), as desired.

It remains to prove the claim that $\dim U_r \leq r$.  Assume $U_r$ is
nonempty. If we restrict $\phi$ to $U_r$, then we obtain a finite map
$U_r \to \phi(U_r)$.  Generically, this restriction has rank $\dim
U_r$, but by definition the rank is at most $r$. Hence, $\dim U_r \leq r$.
\end{proof}

We return to the case of the symplectic vector space $V$.
\begin{corollary} If $A \subseteq \cO_V$ is a graded Poisson subalgebra
containing a regular sequence
 $g_1, \ldots, g_{2n} \in A$ of homogeneous, positively-graded elements,
then
\begin{equation}\label{gidimfla}
\dim \HP_0(A, \cO_V)^* \leq \prod_{i=1}^{2n} (|g_i| - 1).
\end{equation}
\end{corollary}
\begin{proof} This follows immediately if none of the $g_i$ have
  degree one.  On the other hand, if $g_i$ has degree one, then
  $\{g_i, \cO_V\} = \cO_V$ since $\{g_i, -\}$ is a directional
  derivative operator, so $\HP_0(A, \cO_V) = 0$.
\end{proof}
For example, if $G < \GL(X) < \Sp(V)$ is a complex reflection group
and $A = \cO_V^G$, one could take $g_1, \ldots, g_n$ and $g_{n-1},
\ldots, g_{2n}$ to be homogeneous generators of the polynomial
algebras $\cO_X^G$ and $\cO_Y^G$, where $V = X \oplus Y$ is as in the
introduction.  Then, we deduce that $\dim \HP_0(\cO_V^G, \cO_V)^* \leq
\prod_{i=1}^{n} (|g_i| - 1)^2 < \prod_{i=1}^n |g_i|^2 = |G|^2$.  On
the other hand, by Lemma \ref{l:afls-fla}, $\dim \HH_0(\cD_X^G, \cD_X)
= |\{g \in G: g-\Id \text{ is invertible}\}|$, and as explained in the
introduction, this gives a lower bound for $\dim \HP_0(\cO_V^G,
\cO_V)$.  Hence, we deduce
\begin{corollary} \label{c:cr-dim}
If $G < \GL(X) < \Sp(V)$ is a complex reflection
  group, then 
\begin{equation}
  |\{g \in G: g-\Id \text{ is invertible}\}| \leq \dim \HP_0(\cO_V^G, \cO_V)^* < |G|^2.
\end{equation}
\end{corollary}
However, in individual cases, one can do much better than this by
directly computing $\dim R_v$.
\subsection{Applications to representation theory and Poisson geometry}
\label{ss:ncalg}
The material of this subsection is not needed for the rest of the
paper; we include it since it is a natural consequence of the
preceding results.  Let $A = \bigoplus_{i \geq 0} A_i$ be a
nonnegatively graded commutative algebra with a Poisson bracket of
degree $-d < 0$, i.e., $\{A_i, A_j\} \subseteq A_{i+j-d}$.  A
\emph{filtered quantization} is a filtered associative algebra $B =
\bigcup_{i \geq 0} B_{\leq i}$ such that $\gr B = A$ as a commutative
algebra, $[B_{\leq i}, B_{\leq j}] \subseteq B_{\leq (i+j-d)}$, and
$\gr_{i+j-d} [a,b] = \{\gr_i a, \gr_j b\}$ for all $a \in B_{\leq i},
b \in B_{\leq j}$.

Next, given an arbitrary associative algebra $B$ and any
finite-dimensional representation $\rho$ of $B$, the trace functional
$\operatorname{Tr}(\rho): B \to \CC$ annihilates $[B,B]$ and hence
defines an element of $\HH_0(B)^*$.  Given nonisomorphic
finite-dimensional irreducible representations $\rho_1, \ldots,
\rho_m$, the trace functionals $\operatorname{Tr}(\rho_i)$ are
linearly independent (by the density theorem), and hence $\dim
\HH_0(B) \geq m$.  In the situation that $B$ is a filtered
quantization of $A$, one has a canonical surjection $\HP_0(A) \to \gr
\HH_0(B)$ (as in the case of $A=\cO_V^G$ and $B=\cD_X^G$ treated in
the introduction).  Hence, the number of irreducible representations
of $B$ is at most $\dim \HP_0(A)$.

By the material from \cite{ESdm} recalled in the introduction, we conclude:
\begin{corollary}\cite{ESdm}
If $G < \Sp(V)$ is finite, $B$ is an arbitrary 
filtered quantization of $\cO_V^G$, and $v \in V^*$,
then there are at most $\dim R_v$ irreducible finite-dimensional
 representations of $B$.
\end{corollary}
Applying Corollary \ref{gidimfla}, we immediately conclude:
\begin{corollary} If $g_1, \ldots, g_{2n} \in \cO_V^G$ is a regular
  sequence of homogeneous, positively-graded elements, then for every
  filtered quantization $B$ of $\cO_V^G$, there are at most $\prod_i
  (|g_i|-1)$ irreducible finite-dimensional representations.
\end{corollary}
Applying Corollary \ref{c:cr-dim}, we conclude
\begin{corollary}\label{c:cr-irrep}
  If $G$ is a complex reflection group and $B$ a filtered quantization
  of $\cO_V^G$, then there are fewer than $|G|^2$ irreducible
  finite-dimensional representations of $B$.
\end{corollary}
As pointed out after Corollary \ref{c:cr-dim}, in individual cases one
can compute $\dim R_v$ directly, and it is typically much lower than
this. Moreover, $\dim R_v$ is actually a bound on $\dim \HP_0(\cO_V^G,
\cO_V)$, which is in general much larger than the upper bound $\dim
\HP_0(\cO_V^G)$ above.  Finally, again for $G$ a complex reflection
group, when $B$ is a spherical symplectic reflection algebra
quantizing $\cO_V^G$ (see Remark \ref{r:sra} for the notion; note that
these are also called spherical Cherednik algebras in the present case
that $G$ is a complex reflection group), then it is actually known
that there are fewer than $|\text{Irrep}(G)|$ irreducible
finite-dimensional representations of $B$, where $\text{Irrep}(G)$ is
the set of isomorphism classes of irreducible representations of $G$.
This is much better than Corollary \ref{c:cr-irrep}, in these cases.
However, in general, there may exist more general quantizations $B$
than these.

The main goal of this paper is to introduce and apply techniques to
explicitly compute $\HP_0(\cO_V^G)$ in many cases. This in particular
provides the better upper bound $\dim \HP_0(\cO_V^G)$ on the number of
irreducible finite-dimensional representations of quantizations $B$ of
$\cO_V^G$.  These cases include many complex reflection groups,
allowing us to replace the bound $|G|^2$ above by this improved bound.
For example, by Theorem \ref{t:cr-alev} below, applying also Lemma
\ref{l:afls-fla},
\begin{corollary}
  If $G < \GL_2 < \Sp_4$ is one of the complex reflection groups
  $G(m,1,2), G(m,m,2)$, $G(4,2,2), G(6,2,2)$, or $G_4, G_5, G_6, G_8,
  G_9, G_{14}$, or $G_{21}$, then $\HP_0(\cO_V^G) \cong \gr
  \HH_0(\cD_X^G)$ has dimension equal to the number of conjugacy
  classes of elements $g \in G$ such that $g-\Id$ is invertible, i.e.,
  $|\text{Irrep}(G)|-\text{Rank}(G)-1$, where $\text{Rank}(G)$ equals
  the number of conjugacy classes of complex reflections of $G$. Hence, this bounds the number of irreducible
  finite-dimensional representations of every filtered quantization of
  $\cO_{\CC^4}^G$.
\end{corollary}
Note that, in the case $G(m,1,2)$, this is a special case of
\cite[Corollary 1.2.1]{ESsym}, which gives this upper bound in the
case $G = G(m,1,n)$ for arbitrary $m$ and $n$ (as well as for $G = S_n
\ltimes K^n$ for arbitrary $K < \SL_2(\CC)$).  In the other cases,
this bound is new.  Similarly, the bounds $\dim \HP_0(\cO_V^G)$ for
the other groups $G < \GL_2 < \Sp_4$ considered in this paper are new.
\begin{remark}\label{r:sra}
  The filtered quantizations of $\cO_V^G$ include all the
  associated noncommutative spherical symplectic reflection algebras (SRAs),
  defined in \cite{EGsra}. Recall that
  SRAs are certain deformations of $\cO_V \rtimes G$ and spherical SRAs
  are of the form $B = e \widetilde B e$ where $e = \frac{1}{|G|}
  \sum_{g \in G} g \in \CC[G]$ is the symmetrizer element. Noncommutative
  spherical SRAs are those associated to those $\widetilde B$
  obtainable by deforming $\cD_X \rtimes G$ (these form a
  semi-universal family of deformations of $\cD_X \rtimes G$).
\end{remark}
\begin{remark}
  Similarly, one can make a statement about the commutative spherical
  SRAs.  Namely, these are filtered commutative algebras $B$ equipped
  with a Poisson bracket satisfying $\{B_{\leq i}, B_{\leq j}\}
  \subseteq B_{\leq i+j-d}$ such that $\gr B = \cO_V^G$ as a Poisson
  algebra. More generally, if $\gr B = A$ where $B$ is a filtered
  commutative algebra equipped with a Poisson bracket satisfying
  $\{B_{\leq i}, B_{\leq j}\} \subseteq B_{\leq i+j-d}$ and $A$ is
  equipped with the associated graded Poisson bracket of degree $-d <
  0$, then one obtains a canonical surjection $\HP_0(A) \onto \gr
  \HP_0(B)$.  Hence, $\dim \HP_0(B) \leq \dim \HP_0(A)$.  In
  particular, the number of zero-dimensional symplectic leaves (i.e.,
  points whose maximal ideal is a Poisson ideal) of $B$ is dominated
  by $\dim \HP_0(A)$, the same bound as on the number of irreducible
  finite-dimensional representations of filtered quantizations of $A$,
  described in the above results. This is because the zero-dimensional
  symplectic leaves of $B$ all support linearly independent Poisson
  traces on $B$, given by evaluation at that point, and the space of
  Poisson traces on $B$ is the vector space $\HP_0(B)^*$. So, the
  number of zero-dimensional symplectic leaves of commutative
  spherical symplectic reflection algebras associated to $G$ is
  dominated by $\dim \HP_0(\cO_V^G)$, and hence by the same bounds
  described above.
\end{remark}

\section{The case $G < \GL_n < \Sp_{2n}$}\label{glnsec}
As in the introduction, suppose $X$ is a Lagrangian in $V$ and $Y$ a
complementary Lagrangian so that $V = X \oplus Y$. In this section we
restrict to the case that $G < \GL(X) < \Sp(V)$.  As in the
introduction, we may equip $\cO_V$ with a $G$-invariant bigrading, in
which $|X^*| = (1,0)$ and $|Y^*| = (0,1)$. The total degree is the sum
of these degrees. When an element $f$ has bidegree $(a, b)$, we will
also say that $\deg_{X^*} f = a$ and $\deg_{Y^*} f = b$. Similarly,
equip $\cO_{V^*}$ with the bigrading in which $|X|=(-1,0)$ and
$|Y|=(0,-1)$, and when $g \in \cO_{V^*}$ has bidegree $(a,b)$, we say
$\deg_X g = a$ and $\deg_Y g = b$. The total degree is again the sum
of these degrees.

 If we take $v \in X^*$, we can read off $\deg_{Y} g$ (for
 bihomogeneous $g \in \cO_{V^*}$) from its Taylor expansion at $v$: it
 is given by the unique $j \geq 0$ such that there exists $F$ of
 degree $j$ in $Y^*$ such that $F(\partial_{x_1},
 \ldots, \partial_{x_n}, \partial_{y_1}, \ldots, \partial_{y_n})(g)(v)
 \neq 0$.  Moreover, considering \eqref{jvgens}, we see that $J_v$ is
 a bihomogeneous ideal.  Hence, we deduce that
\begin{equation*} 
\dim \{g \in \HP_0(\cO_V^G, \cO_V)^* \mid \deg_{Y}(g) = -j\} \leq \dim \{F \in R_v \mid \deg_{Y^*} F = j\}, \quad \forall v \in X^*, j \geq 0.
\end{equation*}
That is, we get a bound on the Hilbert series of $\HP_0(\cO_V^G,
\cO_V)^*$ with respect to the $Y$-grading, in terms of the
$Y^*$-grading on $R_v$ (for $v \in X^*$).

Next, we note that $\HP_0(\cO_V^G, \cO_V)$ is concentrated in
bidegrees $(i,i), i \geq 0$, since it is annihilated by the action of
the Hamiltonian vector field of $\sum_i x_i y_i$, i.e., the difference
of degrees operator, $\xi_{\sum_i x_i y_i}(g) = (\deg_Y g - \deg_X
g)g$ (for bihomogeneous $g \in \cO_V$). Hence, the total degree of
homogeneous elements of $\HP_0(\cO_V^G, \cO_V)^*$ is always twice the
degree in $Y$ (equivalently, twice the degree in $X$).  We deduce
\begin{theorem}\label{glnthm}
For all $v \in X^*$, 
\begin{equation} \label{glnhsbd}
h(\HP_0(\cO_V^G, \cO_V);t) \leq h((R_v, \deg_Y); t^2).
\end{equation}
Thus, the top degree of $(\HP_0(\cO_V^G, \cO_V)$ is dominated
by twice the top degree of $R_v$ in $Y$.
\end{theorem}
Here, $(R_v, \deg_Y)$ denotes the ring $R_v$ equipped with its grading
by degree in $Y$.

For the purpose of computing the top degree only, one can simplify the 
computation somewhat. Namely, the top degree of $R_v$ in $Y$ is the same
as the top degree of $\overline{R_v} := R_v / (X^*)$. This follows since
 $R_v$ is bihomogeneous.  So we obtain
\begin{equation}
\text{topdeg}(\HP_0(\cO_V^G, \cO_V)) \leq 2 \cdot \text{topdeg}(\overline{R_v}).
\end{equation}
Explicitly, if $v' \in Y$ is the element dual to $v \in X^*$ via the
symplectic pairing, then $\overline{R_v} = \cO_Y / (D_{v'} g_i)_{g_i
  \in \cO_Y}$, where $\cO_Y \subset \cO_V$ are the functions of degree
zero in $X^*$, which we also identify with $\cO_V/(X^*)$.  That is, we
can restrict to those $g_i$ which are only polynomials in the $y_i$.
This has a particular advantage when $G$ is a complex reflection
group, since there $\cO_Y^G$ is a polynomial algebra whose structure
is well known.  We will exploit this below.

\subsection{A bound on top degree using Koszul complexes}\label{glnkcsec}
If we combine Theorem \ref{glnthm} with \eqref{rvkoseqn}, we obtain
\begin{corollary}
Suppose that $h_1, \ldots, h_{2n} \in J_v$ are bihomogeneous and
form a regular sequence, for $v \in X^*$.  Then,
\begin{equation}  \label{glnrshsbd}
h(\HP_0(\cO_V^G, \cO_V);t) \leq \frac{\prod_{i=1}^{2n} (1-t^{2 \deg_Y(h_i)})}{(1-t^2)^{2n}}
\end{equation}
\end{corollary}
The disadvantage of the above corollary is the need to verify the
regular sequence property.  Since the condition $v \in X^*$ is not
generic, we cannot immediately apply Lemma \ref{l:regseq}. To
ameliorate this, we can use an alternative approach, using the
polynomial algebra in only the second half of the variables, $\cO_Y$.
 Namely, rather than computing $R_v$, one can compute
$\overline{R_v} = R_v / (X^*)$ mentioned above, at the price of only bounding
the top degree. Let us write $\overline{R_v} = \cO_Y/\overline{J_v}$ where $\overline{J_v} = J_v / ((X^*) \cap J_v)$.

Thus, if $h_1, \ldots, h_n \in \overline{J_v}$ form a regular sequence
in $\cO_{Y^*}$, then
\begin{equation}
\text{topdeg}(\HP_0(\cO_V^G, \cO_V)) \leq 2 \, \sum_{i=1}^n (|h_i|-1).
\end{equation}
Applying Lemma \ref{l:regseq}, we obtain:
\begin{corollary} \label{glntopdegcor} If $g_1, \ldots, g_n$ are
  homogeneous and form a regular sequence in $\cO_Y^G$, then
\begin{equation}\label{glntopdegeq}
\text{topdeg}(\HP_0(\cO_V^G, \cO_V)) \leq 2 \, \sum_i (|g_i|-2).
\end{equation}
\end{corollary}

\subsection{Complex reflection groups}
In the case of complex reflection groups, $\cO_Y^G$ is
a polynomial algebra generated by homogeneous elements whose degrees
are well known (\cite{STfurg}; see also \cite[Appendix 2]{BMRcrg}). Thus,
in this case, we can apply Corollary \ref{glntopdegcor}
to \emph{generators} $g_1, \ldots, g_n$ of $\cO_Y^G$.
We thus deduce from Corollary \ref{glntopdegcor} explicit bounds on
the top degree of $\HP_0$:
\begin{corollary}\label{crtopdegcor} The top degrees of
  $\HP_0(\cO_V^G, \cO_V)$ for complex reflection groups $G$ are at
  most:
\begin{center}
\begin{tabular}{|cc|cc|cc|}
\hline
$S_{n+1}$: & $n(n-1)$ & $G(m,p,n)$, $m,n > 1$: & $n(n-1)m + 2mn/p - 4n$
& $G(m,1,1)$: & $2(m-2)$ \\ \hline
\end{tabular} \\
\begin{tabular}{|cc|cc|cc|cc|cc|cc|cc|} \hline
$G_4$: & $12$ & $G_5$: & $28$ & $G_6$: & $24$ & $G_7$: & $40$ & $G_8$: & $32$ &
$G_9$: & $56$ & $G_{10}$: & $64$  \\ \hline 
$G_{11}$: & $88$ & $G_{12}:$ & $20$ & $G_{13}$: & $32$ &
$G_{14}$: & $52$ & $G_{15}$: & $64$ & $G_{16}$ & $92$ & $G_{17}$: & $152$
\\ \hline 
$G_{18}$: & $172$ & $G_{19}$: & $232$ & $G_{20}$: & $76$ & $G_{21}$: & $136$ & $G_{22}$: & $56$ & $G_{23}$: & $24$
& $G_{24}$: & $36$ \\ \hline
$G_{25}$: & $42$ & $G_{26}$: & $60$ & $G_{27}$: & $84$
& $G_{28}$: & $40$ & $G_{29}$: & $72$ & $G_{30}$: & $112$ & $G_{31}$: & 
$112$ \\  \hline
\end{tabular} \\
\begin{tabular}{|cc|cc|cc|cc|cc|cc|} \hline
$G_{32}$: & $152$ & $G_{33}$: & $80$ & $G_{34}$: & $240$ & $G_{35}$:  &
$60$  & $G_{36}$: & $112$ & $G_{37}$: & $224$ \\ \hline
\end{tabular}
\end{center}
\end{corollary}
\begin{remark}
  Since the elements $g_1, \ldots, g_n$ can be extended to a
  generating set for $\cO_V$ by elements in the ideal $(X^*)$, e.g.,
  the corresponding generators of $\cO_X^G$, the directional
  derivatives $D_{v'}g_1, \ldots, D_{v'}g_n$ actually generate
  $\overline{J_v} \subseteq \cO_Y$.  Hence, the above bounds coincide
  with those obtained from $R_v$ itself using Theorem \ref{glnthm}, and
  we lose nothing by applying the regular sequence arguments. This is
  in stark contrast to the estimate $\dim R_v < |G|^2$ of Corollary
  \ref{c:cr-dim} (or even
  $\dim R_v \leq \prod_i (|g_i|-1)^2$), where one can
  do much better, in general, by computing $\dim R_v$ directly.
\end{remark}
In the case $S_{n+1}$, the above bound was found by \cite{Matso}, up
to the equivalence of \cite[Theorem 1.5.1]{ReScmat}; in the other
cases, the bounds are new (except for the rank one case, $G(m,1,1)$,
where $\HP_0(\cO_V^G, \cO_V) \cong \HP_0(\cO_V^G)$ is known to have
dimension $2(m-2)$). Using the methods of this paper, we have computed
the actual top degree in the cases of rank $\leq 2$ (with the possible
exception of $G_{18}, G_{19}$) as well as for certain Coxeter groups
of higher rank, which generally differs substantially from the
above. See Remark \ref{r:gmp2-topdeg} for the top
degree in the cases $G(m,p,2)$, and Theorem \ref{t:exst-hilb} for the
top degree in some of the exceptional cases $G_4, \ldots, G_{22}$.

\section{The system of invariant Hamiltonian vector 
fields restricted to a line}\label{matsec}
Now, let $G < \Sp(V)$ and $v \in V^*$ be arbitrary.  Although we know
that elements in $\HP_0(\cO_V^G, \cO_V)^*$ are determined by their
Taylor coefficients by representatives of $R_v$, in general the
grading on $R_v$ is unrelated to the grading on
$\HP_0(\cO_V^G,\cO_V)^*$ (note that $R_v$ is obtained by evaluating at
$v$, which in particular replaces some polynomials on $V^*$ which have
nonzero grading by numbers). To fix this problem, we will use $R_v$ to
construct a local system on the line $\C \cdot v$ and make use of the
Euler vector field, which multiplies by the (correct) degree on
$\HP_0(\cO_V^G, \cO_V)^*$.

Let $f_1, \ldots, f_N$ be a homogeneous basis for $R_v$, and let $F_1,
\ldots, F_N \in \cD_{V^*}$ be differential operators on $V$ such that
$(\gr F_i)|_{T^*_v V^*} \equiv f_i \pmod {J_v}$. Here, restricting
$\gr F_i \in \cO_{T^* V^*}$ to $T^*_v V^*$ means evaluating the
coefficients of the principal symbol $\gr F_i$ of $F_i$ at the point
$v$, obtaining an element of $\cO_{T^*_v V^*} \cong \C[\partial_{x_1},
\ldots, \partial_{x_n}, \partial_{y_1}, \ldots, \partial_{y_n}]$.  For
instance, we can let each $F_i$ be a constant-coefficient differential
operator corresponding to a lift of $f_i$ to $\C[\partial_{x_1},
\ldots, \partial_{x_n}, \partial_{y_1}, \ldots, \partial_{y_n}]$.

\begin{claim}\label{matclaim}
For every $\phi \in \cD_{V^*}$, there exists an operator of
the form $\psi = \sum_i c_i F_i$ for $c_i \in \C$, such that
$\phi(g)|_{\C \cdot v} = \psi(g)|_{\C \cdot v}$ for all $g \in
\HP_0(\cO_V^G,\cO_V)^*$ (i.e., solutions of \eqref{ljvgens}).
\end{claim} 
In other words,
 the derivatives of solutions $g \in \cO_{V^*}$ of \eqref{ljvgens}, evaluated on
the line $\C \cdot v$, depend only on the $F_i(g)$.  

Using the claim, for every $\xi \in \cD_{V^*}$, there exists an $N$ by $N$
matrix $C_\xi \in \Mat_N(\C)$ such that
\begin{equation}\label{cxieqn}
(\xi \circ F_1(g), \ldots, \xi \circ F_N(g))|_{\C \cdot v} = C_\xi (F_1(g), \ldots, F_N(g))|_{\C \cdot v}, \forall g \in \HP_0(\cO_V^G,\cO_V)^*.
\end{equation}
In particular, if $\xi$ is the Euler vector field,
i.e., $\xi(g) = \deg(g) \cdot g$, and if the $F_i$ are homogeneous 
(under the $\C^*$ action on $V$, i.e., $\deg u = -1$ for all $u \in V$, and 
$\deg \partial_{w} = 1$ for all $w \in V^*$) of
degrees $d_1, \ldots, d_N \geq 0$, and 
$g \in \HP_0(\cO_V^G, \cO_V)^*$ is homogeneous, then
\begin{equation}
C_\xi  (F_1(g), \ldots, F_N(g))|_{\C \cdot v} - (d_1 F_1(g), 
\ldots, d_N F_N(g))|_{\C \cdot v} = \deg(g) (F_1(g), \ldots, 
F_N(g))|_{\C \cdot v},
\end{equation}
i.e., $\deg(g)$ is an eigenvalue of the matrix $B_\xi := C_\xi - \Diag(d_1,
\ldots, d_N)$, and $(F_1(g), \ldots, F_N(g))|_{\C \cdot v}$ is an eigenvector. 
Here $\Diag(d_1, \ldots, d_N)$ denotes the diagonal matrix with entries 
$d_1, \ldots, d_N$.
Now, for $\lambda \in \C$ and $C$ a square matrix, let $E_\lambda(C)$ denote the 
$\lambda$-eigenspace of $C$. We obtain
\begin{theorem} \label{matthm}
For arbitrary $v \in V^*$, degree $d_i$ lifts $F_i$ of generators $f_i$ of 
$R_v$ to $\cD_{V^*}$, and $C_\xi$ 
satisfying \eqref{cxieqn} for $\xi$ the Euler vector field,
\begin{equation} \label{cxibdeqn} h(\HP_0(\cO_V^G, \cO_V)^*;t) \leq
  \sum_{i \leq 0} \dim E_i(B_\xi) t^i, \quad B_\xi := C_\xi -
  \Diag(d_1, \ldots, d_N).
\end{equation}
\end{theorem}
It seems that the theorem has the disadvantage that many choices are
involved: in particular, there are many possible choices of the matrix
$C_\xi$. We claim nonetheless that, up to
conjugation, the set of possible $B_\xi$ only depends on the choice of
line $\C \cdot v$, and not on the choice of $f_i$ and $F_i$.  Changing
the $f_i$ and $F_i$ amounts to a combination of linear changes of
basis (which change $C_\xi$ by the corresponding linear changes of
basis), adding homogeneous elements to $F_i$ of the same degree as
$F_i$ which send $\HP_0(\cO_V^G, \cO_V)^*$ to elements which are zero
along $\C \cdot v$ (this does not change $C_\xi$), or multiplying the
$F_i$ by homogeneous polynomials in $\cO_{V^*}$ (which does not change
$B_\xi$).  Hence, the set of possible matrices $B_\xi$ is independent
of these choices up to conjugation, and depends only only the line $\C
\cdot v$. Thus, the same is true for the set of possible bounds (i.e.,
possible polynomials on the RHS of \eqref{cxibdeqn}).

Still, even for fixed $v$, there are in general several nonconjugate
choices of $B_\xi$. This is because, in general, $N$ may exceed $\dim
\HP_0(\cO_V^G, \cO_V)$, and so the coefficients $c_i$ given by Claim
\ref{matclaim} are not uniquely determined.  In practice, however, using only a
single choice of $B_\xi$, the bound one obtains is often equal to the
top degree of $\HP_0(\cO_V^G, \cO_V)$ (or only a few degrees higher),
in contrast to the performance of the methods of \S \ref{glnsec}.

We will explain in \S \ref{algsec} below how to turn this into a
practical algorithm.
\begin{proof}[Proof of Claim \ref{matclaim}]
  Let $I_H := \langle \cD_{V^*}  F_D(\xi_f) \mid f \in \cO_V^G
  \rangle \subset \cD_{V^*}$ be the left ideal generated by the
  Fourier transforms of Hamiltonian vector fields of invariant
  functions.  Note that the solutions $g \in \HP_0(\cO_V^G, \cO_V)^*
  \subset \cO_{V^*}$ are exactly the elements annihilated by $I_H$.

  It is evident that, if $g \in \HP_0(\cO_V^G, \cO_V)^*$, and $\beta
  \in I_H$, then $\beta(g)|_{\C \cdot v} = 0$.  Moreover, $(\gr
  I_H)|_{T^*_v V^*} \supseteq J_v = (\gr(\xi_f): f \in
  \cO_V^G)|_{T^*_v V^*}$ as ideals of $\cO_{T^*_v V^*} =
  \C[\partial_{x_1}, \ldots, \partial_{x_n}, \partial_{y_1},
  \ldots, \partial_{y_n}]$.  Let $I_{v} \subseteq \cO_{V^*}$ be the
  ideal of functions vanishing at $v \in V^*$.  Then, lifts of $f_i$
  to elements $F_i \in \cD_{V^*}$ span $\cD_{V^*}/(I_v \cdot \cD_{V^*}
  + I_H)$, since the latter is filtered and has the associated graded
  vector space $\C[\partial_{x_1},
  \ldots, \partial_{x_n}, \partial_{y_1}, \ldots, \partial_{y_n}] /
  (\gr I_H)|_{T^*_v V}$. Therefore, for every $\phi \in \cD_{V^*}$,
  there exists a linear combination $\psi = \sum_i c_i F_i$ such that
  $\phi-\psi \in I_v \cdot \cD_{V^*} + I_H$, and it follows that
  $\psi(g)|_{\C \cdot v} = \phi(g)|_{\C\cdot v}$ for all $g \in
  \HP_0(\cO_V^G, \cO_V)^*$.
\end{proof}

\subsection{Algorithmic implementation}\label{algsec}
In \cite{ReSc-progs}, we algorithmically construct the $C_\xi$ above.
The first step is to compute the $f_i$ in a way that remembers
additional information.  Normally, one computes generators $f_i$ for
$R_v$ by computing a Gr\"obner basis for $J_v$ with respect to some
ordering of monomials in $\partial_{x_1},
\ldots, \partial_{x_n}, \partial_{y_1}, \ldots, \partial_{y_n}$, e.g.,
the graded reverse-lexicographical ordering (grevlex), whose
definition is recalled below. (Note that we will use monomials to
refer to products of powers of the variables).  We will perform this
computation, following
the Buchberger algorithm, while simultaneously keeping track of lifts
of the Gr\"obner basis elements to elements of $\cD_{V^*}$, as follows.

Recall that the (commutative) Buchberger algorithm works in the
following manner. Fix a polynomial ring $\C[z_1, \ldots, z_n]$. Equip
the monomials with an ordering, such as the grevlex ordering:
$z_1^{a_1} \cdots z_n^{a_n} < z_1^{b_1} \cdots z_n^{b_n}$ if and only
if either $a_1 + \cdots + a_n < b_1 + \cdots + b_n$ or $a_1 + \cdots +
a_n = b_1 + \cdots + b_n$ and, for some $1 \leq i \leq n$, $a_i < b_i$
and $a_{j} = b_{j}$ for all $j > i$.  We require that $g < h$ implies
$fg < fh$ for monomials $f, g$, and $h$, and that $g < h$ when $g$ has lower
total degree than $h$ (which are both true for the grevlex
ordering).

Next, given an ideal $I = (g_1, \ldots, g_m) \subset \C[z_1, \ldots,
z_n]$, we compute a Gr\"obner basis as follows. Assume that the
$g_i$ are all monic, i.e., their leading monomials (with respect to
the monomial ordering) have coefficient one.  Denote the leading
monomial of an element $g$ by $LM(g)$.  
Then, for every pair $i \neq j$, we define the monomial $h :=
\lcm(LM(g_i), LM(g_j))$, and consider the element $g_{ij}$ obtained by
rescaling $\frac{h}{LM(g_i)} \cdot g_i - \frac{h}{LM(g_j)} g_j$ to be
monic (unless it is zero, in which case we set $g_{ij}=0$). If
$g_{ij}=0$, we throw it out. Otherwise, we reduce $g_{ij}$ modulo the
$g_1, \ldots, g_m$, i.e., if $LM(g_k) | LM(g_{ij})$, we replace
$g_{ij}$ with $g_{ij} - \frac{LM(g_{ij})}{LM(g_k)} g_k$. If the result
is zero, we discard it, and otherwise, we rescale it to be monic.  We
then iterate this until we either obtain zero (which we discard) or a
monic polynomial $g$ such that $LM(g_k) \nmid LM(g)$ for all $k$,
which we adjoin to the collection $\{g_1, \ldots, g_m\}$ of generators
of $I$.  (Note that we could have skipped the case $\lcm(LM(g_i),
LM(g_j)) = g_i g_j$, since then we always obtain zero.)  Furthermore,
if $LM(g_i) \mid LM(g_j)$, then we discard $g_j$ (this is the case
where $(g_i, g_j, g_{ij}) = (g_i, g_{ij})$), and vice-versa.  This
process is then repeated until exhaustion, i.e., all pairs of elements
in the generating set have been computed (and no new elements remain
to be added).

In our algorithm, we perform the Buchberger algorithm for $J_v$ while
keeping track, for every generator of $J_v$, of a differential
operator in $I_H$ (the left ideal generated by Hamiltonian vector
fields) lifting the given element.  Namely, we begin with the lifts
$\xi_{f_i}$ of $f_i$ for all $i = 1, 2, \ldots, N$.  Every time we
compute the element $\frac{h}{LM(g_i)} \cdot g_i - \frac{h}{LM(g_j)}
g_j$, for $h = \lcm(LM(g_i), LM(g_j))$, given lifts $\widetilde{g_i},
\widetilde{g_j}$ of $g_i, g_j \in J_v$ to $I_H$, we also compute
$\frac{h}{LM(g_i)} \cdot \widetilde{g_i} - \frac{h}{LM(g_j)}
\widetilde{g_j}$, which is a lift to $I_H$. Here we view
$\frac{h}{LM(g_i)}$ and $\frac{h}{LM(g_j)}$ as constant-coefficient
differential operators. We then rescale and reduce while also keeping
track of the lift to $I_H$.

In the end, we arrive at a Gr\"obner basis $(g_i)$
for $J_v$ together with (noncanonical) lifts $(\widetilde{g_i})$ of
the basis elements to $I_H$.  

Using these lifts, we can reduce $\phi = \xi \circ F_j \in \cD_{V^*}$ to
a linear combination $\psi = \sum_i c_i F_i$ modulo $I_{v} \cdot
\cD_{V^*} + I_H$, as follows: We work in $\cD_{V^*} / (I_v \cdot \cD_{V^*})$,
which identifies with $\cO_{T^*_v V^*} \cong \C[\partial_{x_1},
\ldots, \partial_{x_n}, \partial_{y_1}, \ldots, \partial_{y_n}]$
as a vector space.  Define $\overline{I_H} := (I_H + I_v \cdot \cD_{V^*})
/ (I_v \cdot \cD_{V^*})$, which is a vector subspace.  Under the above
identification, $\overline{I_H}$ is filtered (by order of differential
operators), and $\gr \overline{I_H} \supseteq J_v$.  Let
$\overline{\widetilde{g_i}} \in \overline{I_H}$ be the image of $\widetilde{g_i} \in I_H$
under this quotient.  Then, $\gr \overline{\widetilde{g_i}} = g_i$.
We may now reduce $\overline{\phi} \in \cD_{V^*} / (I_v \cdot \cD_{V^*})$
modulo $\overline{I_H}$ by iteratively reducing $\gr \overline{\phi}$
modulo $J_v$, such that every time we subtract $g \cdot g_i$ from $\gr
\overline{\phi}$ for $g \in \C[\partial_{x_1},
\ldots, \partial_{x_n}, \partial_{y_1}, \ldots, \partial_{y_n}]$
a constant-coefficient differential operator, we simultaneously
subtract $g \cdot \overline{\widetilde{g_i}}$ from $\overline{\phi}$.

\section{Computational results}\label{compsec}
We developed computer programs in Magma \cite{ReSc-progs} to compute
$\HP_0(\cO_V^G, \cO_V)$ using the above theory.  First, we wrote
programs which compute $\HP_0(\cO_V^G, \cO_V)$ (together with its
grading and $G$-structure) up to a specified degree. Then, we wrote
programs which compute the bounds of Theorems \ref{glnthm} and
\ref{matthm}.

It turns out that, in practice, the bound produced by Theorem
\ref{matthm} (using the matrix $B_\xi$) is much sharper than that of
Theorem \ref{glnthm} (which is only applicable to the case $G < \GL(X)
< \Sp(V)$).  In particular, in most cases we tested, the top integer
eigenvalue of $-B_\xi$ (for appropriate $v \in V^*$) was in fact equal
to the top degree of $\HP_0(\cO_V^G, \cO_V)$ (recall that the degrees
of $\HP_0(\cO_V^G, \cO_V)$ are nonpositive, which is why we have a
minus sign in $-B_\xi$).  This is good because it can also be applied
to arbitrary $G < \Sp(V)$. The downside is that the computation
required can be much slower, and sometimes too slow.

In the case of groups $G < \GL(X) < \Sp(V)$, we actually use both
techniques: first we apply \S \ref{glnsec} to compute the (generally
less sharp) bound $2 \cdot \text{topdeg}(\overline{R_v})$ on the top degree;
this is usually very fast, and for complex reflection groups the
result is already in Corollary \ref{crtopdegcor}. Next, we compute
$-B_\xi$ and its eigenvalues working over a prime field $\FF_p$ for
$p$ larger than the first bound.  This can be effectively computed in
some cases where it is not over a number field. Although, in theory,
this could produce a less sharp bound than over a number field, in
practice, it is quite effective, and one obtains a useful bound (often
the actual top degree).

Finally, once we have this bound on degree, we use our programs to
explicitly compute $\HP_0(\cO_V^G,\cO_V)$ up to that top degree,
working over a number field (either the field of definition of $G$,
generally a cyclotomic field, or a smaller subfield containing the
coefficients of generators of the invariant ring, over which one can
therefore define $\cO_V^G$: for example, for some of the
exceptional Shephard-Todd groups of rank two, one can compute
generators of $\cO_V^G$ with rational coefficients even though the
generators of $G$ do not have rational coefficients).  If this is too
slow, one could work over a prime field $\FF_p$ containing primitive
$|G|$-th roots of unity, although then the result would technically
only yield an upper bound for the ($G$-graded) Hilbert series of
$\HP_0(\cO_V^G, \cO_V)$ (in practice, one will probably get the right
answer if the prime $p$ is large).  However, if one
obtains in this way a group $\HP_0(\FF_p[V]^G, \FF_p[V])$ of dimension
$|\{g \in G \mid (g-\operatorname{Id}) \text{ is invertible}\}| = \dim
\HH_0(\cD_X^G, \cD_X)$ (cf.~Lemma \ref{l:afls-fla}), then this must be the correct dimension since
this is a lower bound for $\dim \HP_0(\cO_V^G, \cO_V)$, and therefore
$\HP_0(\cO_V^G, \cO_V) \cong \gr \HH_0(\cD_X^G, \cD_X)$.
\subsection{Subgroups of $\SL_2(\CC)$}
In \cite{AL}, the groups $\HP_0(\cO_V^G)$ were computed for $V =
\CC^2$ and $G < \Sp(V) = \SL_2(\CC)$ a finite subgroup (for an
alternative computation, one can specialize \cite{ESsym} to the rank
one case).  The associated varieties $V/G$ are well known and are
called Kleinian singularities. It then follows from Lemma \ref{l:afls-fla}
(the main result of \cite{AFLS}) that
$\HP_0(\cO_V^G) \cong \gr \HH_0(\cD_X^G)$. 

In this subsection, we extend this by computing $\HP_0(\cO_V^G,
\cO_V)$. Our main result is Theorem \ref{t:sl2} below, which we expand on in the subsequent sections.
\begin{definition}
Given a graded vector space $K$, let $K_\ev$ denote the span of the even-graded homogeneous elements of $K$.
\end{definition}
The following elementary lemma explains our interest in the even-graded subspace:
\begin{lemma} Let $V$ be an arbitrary finite-dimensional symplectic vector
space and $G < \Sp(V)$ finite. Then, $\gr \HH_0(\cD_X^G, \cD_X)$ is concentrated in even degrees.
\end{lemma}
\begin{proof}
First suppose that $-\Id \in G$.  Since $-\Id$ is central, it acts trivially on $\C[G]_\ad$ and hence on $\HH_0(\cD_X^G, \cD_X)$ by Lemma \ref{l:afls-fla}.  Since the action of $-\Id$ on $\gr \HH_0(\cD_X^G, \cD_X)$ is by
 $(-1)^{\deg}$, this implies that it is concentrated in even degrees. 

 In the general case, let $K := \langle G, -\Id \rangle$.  Then,
 $\HH_0(\cD_X^G, \cD_X)$ is a quotient of $\HH_0(\cD_X^K, \cD_X)$, so
 this also holds on the level of associated graded vector spaces.
 Therefore, by the above paragraph, $\gr \HH_0(\cD_X^G, \cD_X)$ is
 concentrated in even degrees.
\end{proof}
Let $\widetilde{D_{m}}$ denote the dicyclic subgroup of order $2m$
(for $m$ even), which is the inverse image of the dihedral subgroup
$D_m$ of $\operatorname{\mathsf{SO}}(3,\RR)$ under the double cover by
$\mathsf{SU}(2,\CC)$.  It is well known (the ``McKay correspondence'')
that all finite subgroups of $\SL_2(\CC)$ are either cyclic, dicyclic,
or one of the three exceptional groups $\widetilde A_4, \widetilde
S_4$, and $\widetilde A_5$, which are the preimages of the
tetrahedral, octahedral, and icosahedral rotation subgroups of
$\operatorname{\mathsf{SO}}(3,\RR)$ in $\mathsf{SU}(2,\CC) <
\SL_2(\CC)$ under the double cover $\mathsf{SU}(2,\CC) \onto \operatorname{\mathsf{SO}}(3,\RR)$.

By the McKay correspondence, the cyclic, dicyclic, and exceptional
groups correspond to the simply-laced extended Dynkin diagrams of
types $\widetilde A, \widetilde D$, and $\widetilde E$, respectively:
the vertices are the irreducible representations of the group, and
given an irreducible representation, the decomposition of its tensor
product with the defining representation $\CC^2$ into irreducibles is
given by the vertices adjacent to the one corresponding to the
original irreducible representation.
\begin{theorem}\label{t:sl2}
If $G < \SL_2(\CC)$ is finite, then the composition $\HP_0(\cO_V^G, \cO_V)_\ev \into \HP_0(\cO_V^G, \cO_V) \onto \gr \HH_0(\cD_X^G, \cD_X)$ is an isomorphism.  The Hilbert series of $h(\HP_0(\cO_V^G, \cO_V);t)$ is given by
\begin{gather}
1 + t^2 + \cdots + t^{2(m-2)}, \quad G \cong \ZZ/m; \label{e:hp0-an}\\
1 + (2t + 3t^2 + 2t^3 + \cdots + 3t^{m-2})
 + 2t^m + (t^{m+2} + t^{m+4} + \cdots + t^{2m-4}) + t^{2m}, \quad G \cong \widetilde{D_{m}}; \label{e:hp0-dn} \\
1+2t+3t^2+4t^3+5t^4+4t^5+4t^6+2t^7+4t^8+3t^{10}+t^{12}+t^{14}+t^{20}, \quad G \cong \widetilde{A_4}; 
\end{gather}
\begin{multline}
1+2t+3t^2+4t^3+5t^4+6t^5+7t^6+6t^7+6t^8+6t^9 \\ + 6t^{10}+4t^{11}+6t^{12}+2t^{13}+ 4t^{14}+3t^{16}+3t^{18}+t^{20}+t^{24}+t^{32}, \quad G \cong \widetilde{S_4}; 
\end{multline}
\begin{multline}
1+2t+3t^2+4t^3+5t^4+6t^5+7t^6+8t^7+9t^8+10t^9+11t^{10}+10t^{12}+10t^{13}+10t^{14}\\+10t^{15}+10t^{16}+10t^{17}+10t^{18}+8t^{19}+10t^{20}+6t^{21}+6t^{22}+4t^{23}+6t^{24}+2t^{25}\\+6t^{26}+5t^{28}+3t^{30}+t^{32}+3t^{34}+t^{36}+t^{44}+t^{56}, \quad G \cong \widetilde{A_5}, \label{e:hp0-e8}
\end{multline}
and $h(\HP_0(\cO_V^G);t)$ is given by \eqref{e:hp0-an} when $G \cong \ZZ/m$, and
\begin{gather}
(1 + t^4 + \cdots + t^{2m}) + t^m, \quad G \cong \widetilde{D_{m}};  \label{e:hp0-dn-invts} \\
1+t^6+t^8+t^{12}+t^{14}+t^{20}, \quad G \cong \widetilde{A_4}; \\
1+t^8+t^{12}+t^{16}+t^{20}+t^{24}+t^{32}, \quad G \cong \widetilde{S_4}; \\
1+t^{12}+t^{20}+t^{24}+t^{32}+t^{36}+t^{44}+t^{56}, \quad G \cong \widetilde{A_5}. \label{e:hp0-e8-invts}
\end{gather}
\end{theorem}
By the lemma, the composition $\HP_0(\cO_V^G,\cO_V)_\ev \to \gr
\HH_0(\cD_X^G, \cD_X)$ is always a surjection. The fact that it is
injective follows from the explicit formulas for Hilbert series above,
since this together with Lemma \ref{l:afls-fla} shows that the
dimensions are equal.  Thus, below, we restrict our attention to
proving \eqref{e:hp0-an}--\eqref{e:hp0-e8}.

On the other hand, the map $\HP_0(\cO_V^G, \cO_V) \onto \gr
\HH_0(\cD_X^G, \cD_X)$ itself is not injective when $G< \SL_2(\CC)$ is
not abelian, since $\HP_0(\cO_V^G, \cO_V)$ is not concentrated in even
degrees. Nonetheless, by the above formulas (or \cite{AL}) together
with Lemma \ref{l:afls-fla}, the restriction to invariants,
$\HP_0(\cO_V^G) \onto \gr \HH_0(\cD_X^G)$, is an isomorphism.
\begin{remark}
  The above gives examples where $\HP_0(\cO_V^G, \cO_V)$ is not
  concentrated in even degrees, but $\HP_0(\cO_V^G)$ is.  It is
  natural to ask for an example where $\HP_0(\cO_V^G)$ itself is not
  concentrated in even degrees. We construct such examples in Appendix
\ref{s:nontriv-cubic}. 
\end{remark}
\begin{remark}
  The fact that $\HP_0(\cO_V^G, \cO_V)_\ev \cong \gr \HH_0(\cD_X^G,
  \cD_X)$ is quite special to the above case.  For many groups $G$
  (such as many examples discussed below), $\HP_0(\cO_V^G) \ncong \gr
  \HH_0(\cD_X^G)$ and the former is concentrated in even degrees (in
  the cases below, $G < \GL(X) < \Sp(V)$, so $\HP_0(\cO_V^G, \cO_V)$
  itself is automatically concentrated in even degrees, by the
  discussion at the beginning of \S \ref{glnsec}). There are also
  examples where $\HP_0(\cO_V^G) \cong \gr \HH_0(\cD_X^G)$ but still
  $\HP_0(\cO_V^G, \cO_V)_\ev \ncong \gr \HH_0(\cD_X^G, \cD_X)$. For
  example, this holds when $G$ is the complex reflection group
  $G(4,2,2)$ or $G(6,2,2)$ as discussed below.
\end{remark}
As already remarked, the formulas
\eqref{e:hp0-dn-invts}--\eqref{e:hp0-e8-invts} were first computed in
\cite{AL}, but we include them since they follow directly from the
(apparently new) formulas \eqref{e:hp0-an}--\eqref{e:hp0-e8} of the
theorem.\footnote{As is well-known,
  \eqref{e:hp0-dn-invts}--\eqref{e:hp0-e8-invts} can be more compactly
  described as $\sum_i t^{2(m_i-1)}$, where $m_i$ are the Coxeter
  exponents of the root system corresponding to the group by the McKay
  correspondence (type $A_{m-1}$ in the case of $\ZZ/m$, type
  $D_{m/2}$ in the dicyclic case, and types $E_6, E_7$, and $E_8$ in
  the exceptional cases).}  Note that, when $G$ is abelian (and hence
cyclic since $V = \CC^2$), by Lemma \ref{l:abelian} below,
$\HP_0(\cO_V^G, \cO_V) = \HP_0(\cO_V^G)$, so \eqref{e:hp0-an} also
follows from \cite{AL}.  Thus, we do not need to discuss the cyclic
case at all, but we do so anyway since the computation is short and
simple.

Let us write $\cO_V = \CC[x,y]$ with $\{x,y\}=1$. Using the
symplectic form, $V \cong \Span(x,y)$, and let us write matrices
according to their action on the basis pulled back from $(x,y)$.  We will use the following elementary lemma, which holds for arbitrary symplectic $V$ and $G < \Sp(V)$:
\begin{lemma}\label{l:pg}
Let $(g_i)$ be a collection of Poisson generators of $\cO_V^G$. Then
$\{\cO_V^G, \cO_V\}$ is the sum of the subspaces $\{g_i, \cO_V\}$.
\end{lemma}
\begin{proof}
It suffices to show that, for all $f, g \in \cO_V^G$ and all $h \in \cO_V$, that $\{fg, h\}$ and $\{\{f,g\},h\}$ are subspaces of $\{f, \cO_V\} + \{g, \cO_V\}$.  This follows from the identities
\[
\{fg, h\} = \{f, gh\} + \{g, fh\}, \quad \{\{f,g\}, h\} = \{f,
\{g,h\}\} - \{g, \{f,h\}\}. \qedhere
\]
\end{proof}

\subsubsection{Cyclic subgroups}

Suppose $G \cong \ZZ/m$. We give a short, self-contained proof of
\begin{theorem}\cite{AL} $h(\HP_0(\cO_V^G, \cO_V);t) = 1 + t^2 + \cdots + t^{2(m-2)}$, and $G$ acts trivially. Moreover, a basis is obtained by the images of the elements $x^a y^a$ for $0 \leq a \leq m-2$.
\end{theorem}
Up to conjugation, $G = 
\biggl\langle \begin{pmatrix} e^{2\pi i/m} & 0 \\ 0 & e^{-2\pi
    i/m} \end{pmatrix} \biggr\rangle$.  The ring $\cO_V^G$ is
generated by the elements $xy, x^m$, and $y^m$.  It is Poisson
generated by the first two elements.

Therefore, by Lemma \ref{l:pg}, we only need to compute $\{xy,
\cO_V\}$ and $\{x^m, \cO_V\}$. The former is spanned by all
monomials of unequal degrees in $x$ and $y$.  The latter is spanned by monomials of degree $\leq m-1$ in $x$.  Hence, a basis for $\HP_0(\cO_V^G, \cO_V)$ is given by $(1, xy, \ldots, x^{m-2} y^{m-2})$.  This recovers the theorem.

\subsubsection{Dicyclic subgroups}
By the classification of finite subgroups of $\SL_2(\CC)$ recalled
above, the other infinite family of subgroups is that of the dicyclic
groups, which are given up to conjugation by
\[
G = \biggl\langle \begin{pmatrix} e^{2\pi i/m} & 0 \\ 0 & e^{-2\pi
    i/m} \end{pmatrix}, \begin{pmatrix} 0 & -1 \\ 1 & 0 \end{pmatrix}
\biggl\rangle,
\]
for $m$ even. Let $\rho_0$ denote the trivial representation of $G$,
$\rho_1$ the nontrivial one-dimensional representation which vanishes on
the diagonal elements, $\rho_3$ and $\rho_4$ the other one-dimensional representations (in either order), $\tau_1$ the standard $2$-dimensional representation,
 and $\tau_j$ the irreducible two-dimensional representation in which the diagonal elements act through their $j$-th powers (for $1 \leq j \leq m/2-1$).

The goal of this section is to prove
\begin{theorem}
As a graded $G$-representation, $H := \HP_0(\cO_V^G, \cO_V)$ is given by
\begin{gather*}
h(\Hom_G(\rho_0, H);t) = (1 + t^4 + \cdots + t^{2m}) + t^m; \quad
h(\Hom_G(\rho_1, H);t) = (t^2 + t^6 + \cdots + t^{2m-6}) + t^m; \\
h(\Hom_G(\rho_2,H);t) = h(\Hom_G(\rho_3,H);t) = t^{m/2}; \\
h(\Hom_G(\tau_1, H);t) = t; \quad h(\Hom_G(\tau_j, H);t) = t^j + t^{m-j}, \quad 2 \leq j \leq m/2-1.
\end{gather*}
\end{theorem}
The invariant ring $\cO_V^G$ in generated by $x^2y^2, x^m+y^m$, and
$xy(x^m-y^m)$.  The first two of these are Poisson
generators.  By Lemma \ref{l:pg}, we therefore only need to compute
$\{x^2y^2, \cO_V\}$ and $\{x^m+y^m, \cO_V\}$.

First, $\{x^2 y^2, \cO_V\}$ is spanned by $\{x^2 y^2, x^a y^b \} = 2(b-a) x^{a+1} y^{b+1}$. This is the span of all monomials of unequal positive
degrees in $x$ and $y$.

Next, $\{x^m+y^m, \cO_V\}$  is spanned by $\{x^m+y^m, x^a y^b \} = bm x^{a+m-1} y^{b-1} - am x^{a-1} y^{b+m-1}$.  Up to the previous span, this is the same as the span of the monomials $x^a y^b$ with either $a \geq m-1$ or $b \geq m-1$, with
the exception of the pairs $(a,b) \in \{(m,0), (0,m), (2m,0), (m,m), (0,2m)\}$, where we obtain the elements
\[
x^m - y^m, \quad mx^{2m} - m(m+1)x^my^m, \quad m(m+1)x^my^m - my^{2m}.
\]
As a result, the following elements map to a graded basis of $\HP_0(\cO_V^G, \cO_V)$:
\begin{equation}
(1) \cup (x^a, y^a, x^ay^a)_{1 \leq a \leq m-2} \cup (x^m+y^m) \cup ((m+1)(x^{2m} + y^{2m}) + x^my^m).
\end{equation}
Moreover, the span of these elements is $G$-invariant, and the theorem
follows easily.

\subsubsection{Exceptional subgroups}
By computer programs in Magma, we computed for the exceptional
subgroups the graded representations $\HP_0(\cO_V^G, \cO_V)$.  In this
case, one can prove that the answer is correct using only the bound on
dimension, $\dim R_v$, from the introduction, for a particular choice
of $v$, since for $G < \SL_2$, $\gr (\xi_{h_i}) = (\gr \xi_{h_i})$, as
$h_i$ ranges over generators of $\cO_V^G$.  Just to double-check, we
also employed the programs using the method of \S \ref{matsec} (since
$\dim R_v = \dim \HP_0(\cO_V^G, \cO_V)$ in this case, this yields
precisely the correct Hilbert series, i.e., \eqref{cxibdeqn} is an
equality.)

Label the representations of $G \in \{\widetilde{A_4}, \widetilde{S_4}, \widetilde{A_5}\}$, corresponding to the McKay graph $E_m$, 
by $\rho_0, \ldots, \rho_m$, with $\rho_0$ the trivial representation, according to Figure \ref{f:e6-8}. Our indexing follows Magma (in particular, indices increase with the dimension of the irreducible representation).
\begin{figure}
\includegraphics[width = 0.8\textwidth]{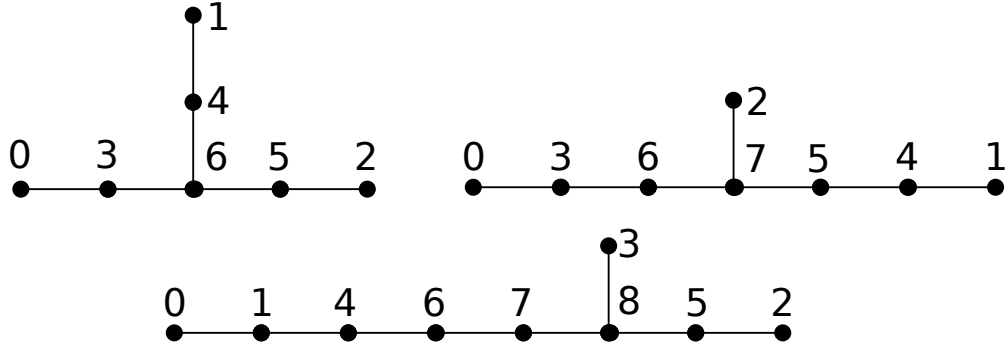}
\caption{Labels of the irreducible representations of $\widetilde{A_4}, \widetilde{S_4}, \widetilde{A_5} < \SL_2(\CC)$ in terms of the McKay graphs $\widetilde{E_6}, \widetilde{E_7}, \widetilde{E_8}$, respectively.  The defining $2$-dimensional representation is the second from the left in all cases; tensoring a representation by
this representation yields the direct sum of all adjacent representations.\label{f:e6-8}}
\end{figure}
\begin{theorem}
The graded $G$-structure of $H = \HP_0(\cO_V^G, \cO_V)$ is given by:
\begin{enumerate}
\item[$G=\widetilde{A_4}$:]
$h(\Hom_G(\rho_0, H);t)=1+t^6+t^8+t^{12}+t^{14}+t^{20};$ \\
$h(\Hom_G(\rho_1, H);t)=h(\Hom_G(\rho_2, H);t)= t^4;$ \\
$h(\Hom_G(\rho_3, H);t)=t+t^7;$ \\
$h(\Hom_G(\rho_4, H);t)=h(\Hom_G(\rho_5, H);t)=t^3 + t^5;$ \\
$h(\Hom_G(\rho_6, H);t)=t^2+t^4+t^6+t^8+t^{10}$.
\item[$G=\widetilde{S_4}$:] $h(\Hom_G(\rho_0, H);t)=1+t^8+t^{12}+t^{16}+t^{20}+t^{24}+t^{32}$;  $\quad$
$h(\Hom_G(\rho_1, H);t) = t^6+t^{14}$; \\
$h(\Hom_G(\rho_2, H);t) = t^4+t^8+t^{12}+t^{16}$; $\quad$
$h(\Hom_G(\rho_3, H);t) = t+t^9$; \\
$h(\Hom_G(\rho_4, H);t) = t^5+t^7+t^{13}$; $\quad$
$h(\Hom_G(\rho_5, H);t) = t^4+t^6+t^8+t^{12}$; \\
$h(\Hom_G(\rho_6, H);t) = t^2+t^6+2t^{10}+t^{14}+t^{18}$; $\quad$
$h(\Hom_G(\rho_7, H);t) = t^3+t^5+t^7+t^9+t^{11}$.
\item[$G=\widetilde{A_5}$:]
$h(\Hom_G(\rho_0, H);t)=1+t^{12}+t^{20}+t^{24}+t^{32}+t^{36}+t^{44}+t^{56}$;  $\quad$
$h(\Hom_G(\rho_1, H);t)=t+t^{13}+t^{25}$; \\
$h(\Hom_G(\rho_3, H);t)= t^6+t^{10}+t^{14}+t^{18}+t^{22}+t^{26}+t^{30}$; $\quad$
$h(\Hom_G(\rho_2, H);t)=t^7+t^{13}+t^{19}$; \\
$h(\Hom_G(\rho_4, H);t)=t^2+t^{10}+t^{14}+t^{18}+t^{22}+t^{26}+t^{34}$; \\
$h(\Hom_G(\rho_5, H);t)=t^6+t^8+t^{12}+t^{14}+t^{18}+t^{20}$; \\
$h(\Hom_G(\rho_6, H);t)=t^3+t^9+t^{11}+t^{15}+t^{17}+t^{23}$; \\
$h(\Hom_G(\rho_7, H);t)=t^4+t^8+t^{10}+t^{12}+2t^{16}+t^{20}+t^{24}+t^{28}$; \\
$h(\Hom_G(\rho_8, H);t)=t^5+t^7+t^9+t^{11}+t^{13}+t^{15}+t^{17}+t^{19}+t^{21}$.
\end{enumerate}
\end{theorem}

\subsection{Coxeter groups of rank $\leq 3$ and $A_4, B_4=C_4$, and $D_4$}\label{ss:cox}
\begin{theorem}
For every Coxeter group $G< \GL(X) < \Sp(V)$ of rank $\leq 3$, $\HP_0(\cO_V^G,
\cO_V) \cong \gr \HH_0(\cD_X^G,\cD_X)$. The resulting Hilbert series is
\begin{gather*}
A_1: 1; \quad A_2: 1+t^2; \quad A_3: 1+3t^2+2t^4; \\
B_2=C_2: 1+t^2+t^4; \quad B_3=C_3: 1+3t^2+6t^4+4t^6+t^8;\\
H_3: 1+3t^2+6t^4+10t^6+15t^8+9t^{10}+t^{12}; \quad
I_2(m): 1 + t^2 + \cdots + t^{2(m-2)}. 
\end{gather*}
Also, for types $A_4, B_4=C_4$, and $D_4$,  $\HP_0(\cO_V^G,
\cO_V) \cong \gr \HH_0(\cD_X^G,\cD_X)$ holds. The resulting Hilbert series are
\begin{gather*}
A_4: 1+6t^2+10t^4+6t^6+t^8; \quad D_4: 1+6t^2+20t^4+16t^6+2t^8; \\
B_4=C_4: 1+6t^2+20t^4+31t^6+28t^8+15t^{10}+4t^{12}. 
\end{gather*}
The Hilbert series of $\HP_0(\cO_V^G) \cong \HH_0(\cD_X^G)$ in
all of these cases are
\begin{gather*}
A_1, A_2, A_3, A_4: 1;  \quad D_4: 1+t^4+t^8; \\
B_2=C_2: 1+t^4;  \quad B_3=C_3: 1+t^4+t^8; \quad B_4=C_4: 1+t^4+2t^8+t^{12}; \\
H_3: 1+t^4+t^8+t^{12}; \quad
I_2(m): 1 + t^4 + \cdots + t^{4 \lfloor (m-2)/2 \rfloor}. 
\end{gather*}
\end{theorem}
\begin{remark}
Partial computer tests have shown that $\HP_0(\cO_V^G, \cO_V) \ncong \gr \HH_0(\cD_X^G, \cD_X)$ for $G = F_4$, although we do not know whether the identity holds on the level of invariants.
\end{remark}
\begin{remark}
  The surjection $\HP_0(\cO_V^G, \cO_V) \onto \gr \HH_0(\cD_X^G,
  \cD_X)$ is not, in general, an isomorphism for Coxeter groups of
  rank $\geq 5$. Via the equivalence of \cite[Theorem 1.5.1]{ReScmat},
  \cite[8.6]{Matso} (see also \cite[Example 1.6.1]{ReScmat}) shows
  that $\HP_0(\cO_V^G, \cO_V) \ncong \gr \HH_0(\cD_X^G, \cD_X)$ when
  $G \cong S_{n+1}$ is a Weyl group of type $A_{n}$ for $n \geq 5$
  (but, $HP_0(\cO_V^G) \cong \gr \HH_0(\cD_X^G)$ for all types $A_n$
  by \cite{ESweyl}).  Also, by \cite{ESweyl}, $\HP_0(\cO_V^G) \ncong
  \gr \HH_0(\cD_X^G)$ when $G$ is a Weyl group of type $D_n$ for $n
  \geq 5$.
\end{remark}
\begin{question} In the cases $F_4, H_4, E_6, E_7$, and $E_8$, does
  $\HP_0(\cO_V^G) \cong \gr \HH_0(\cD_X^G)$ hold? If so, in any case
  (except $F_4$), does $\HP_0(\cO_V^G, \cO_V) \cong \gr \HH_0(\cD_X^G,
  \cD_X)$ hold?
\end{question}

\subsection{Complex reflection groups of rank two}
\begin{theorem}\label{t:cr-alev}
Of the complex reflection groups of rank two, the ones such that $\HP_0(\cO_V^G, \cO_V) \cong \gr \HH_0(\cD_X^G, \cD_X)$ are exactly $S_3, G(m,1,2), G(m,m,2),
G_4, G_6, G_8$, and $G_{14}$.  The additional groups such that $\HP_0(\cO_V^G) \cong \gr \HH_0(\cD_X^G)$ are $G(4,2,2), G(6,2,2), G_5, G_9$, and $G_{21}$.
\end{theorem}
We also compute the relevant Hilbert series, where $\HP_0$ and $\HH_0$
coincide.  For the case $S_3$, this is given in the previous section,
and the $G(m,p,2)$ case is treated in \S \ref{s:gmp2}, where we also
prove the above theorem in this case.  For the exceptional cases, we
used Magma programs and the techniques of \S$\!$\S \ref{glnsec} and
\ref{matsec} to compute $\HP_0(\cO_V^G, \cO_V)$ for all $G_4, \ldots,
G_{22}$ except $G_{18}$ and $G_{19}$, and computed enough of
$\HP_0(\cO_V^G)$ for the cases $G_{18}$ and $G_{19}$ to prove that
  $\HP_0(\cO_{V}^G) \ncong \gr \HH_0(\cD_X^G)$ (in fact, it seems
  we computed all of $\HP_0(\cO_V^G)$, but we could not prove
  it).  We give the results in the cases where the isomorphism holds:
\begin{theorem}\label{t:exst-hilb}
The Hilbert series of $\HP_0(\cO_{V}^G) \cong \gr \HH_0(\cD_X^G)$ for
the exceptional Shephard-Todd groups $G_4, G_5,
  G_6, G_8, G_9, G_{14}$, and $G_{21}$ where this holds are:
\begin{gather*}
G_4:  1+t^2+t^4+t^8; \quad G_5: 1+t^2+t^4+2t^6+3t^8+2t^{10}+2t^{12}+2t^{14}+t^{16}+t^{20}; \\
G_6: 1+t^2+t^4+t^6+2t^8+t^{10}+t^{12}+t^{14}+t^{16}; \\
G_8: 1+t^2+t^4+t^6+2t^8+t^{10}+2t^{12}+t^{14}+t^{16}+t^{20}; \\
G_9: 1+t^2+t^4+t^6+2t^8+2t^{10}+3t^{12}+2t^{14}+3t^{16}+2t^{18}+3t^{20}+t^{22}+2t^{24}+t^{26}+t^{28}+t^{32}; \\
G_{14}:1+t^2+t^4+t^6+2t^8+t^{10}+2t^{12}+2t^{14}+2t^{16}+t^{18}+2t^{20}+t^{22}+t^{24}+t^{26}+t^{28};
\end{gather*}
\begin{multline*}
G_{21}: 1 + t^2+t^4+t^6+t^8+t^{10}+2t^{12}+2t^{14}+2t^{16}+2t^{18}+3t^{20}+2t^{22}+3t^{24}+3t^{26}+3t^{28}\\+2t^{30}+3t^{32}+2t^{34}+3t^{36}+2t^{38}+2t^{40}+t^{42}+2t^{44}+t^{46}+t^{48}+t^{50}+t^{52}+t^{56}.
\end{multline*}
The Hilbert series of $\HP_0(\cO_V^G, \cO_V) \cong \gr \HH_0(\cD_{X}^G, \cD_X)$ in
the cases $G_4, G_6, G_8$, and $G_{14}$ where this holds are
\begin{gather*}
G_4: 1+4t^2+6t^{14}+3t^6+t^8; \quad G_6: 1+4t^2+9t^4+7t^6+5t^8+4t^{10}+t^{12}+t^{14}+t^{16}; \\
G_8: 1+4t^2+9t^4+16t^6+17t^8+13t^{10}+10t^{12}+5t^{14}+t^{16}+t^{20}; \\
G_{14}: 1+4t^2+9t^4+16t^6+22t^8+18t^{10}+15t^{12}+11t^{14}+7t^{16}+6t^{18}+2t^{20}+t^{22}+t^{24}+t^{26}+t^{28}.
\end{gather*}
\end{theorem}

\section{Abelian subgroups of $\Sp_4$}\label{s:sp4abel}
In this section, we describe $\HP_0(\cO_V^G, \cO_V)$ in the case that
$V=\C^4$ and $G$ is an abelian subgroup of $\Sp_4$.  By the following
elementary lemma, it suffices to assume that $G < (\C^\times)^2 <
\GL_2 < \Sp_4$, and moreover, in this case, $\HP_0(\cO_V^G,
\cO_V)=\HP_0(\cO_V^G)$:
\begin{lemma}\label{l:abelian}
  Let $G < \Sp_{2n}$ be a finite abelian subgroup. Then, up to
  conjugation, $G < (\C^\times)^n < \GL_n < \Sp_{2n}$ is a subgroup of
  diagonal matrices. Moreover, $G$ acts trivially on $\HP_0(\cO_{\C^{2n}}^G,
  \cO_{\C^{2n}})$.
\end{lemma}
\begin{proof}
  To prove the first statement, we proceed inductively. There must
  exist a common eigenvector $v_1 \in \C^{2n}$ for $G$. Set $V_1 :=
  \Span(v_1)$.  Since $G < \Sp_{2n}$ and $G$ stabilizes $V_1$, it also
  stabilizes $V_1^\perp$. If $\dim V_1^\perp > \dim V_1$, pick another
  common eigenvector $v_2 \in \dim V_1^\perp$ not in $V_1$, and set
  $V_2 := \Span(v_1,v_2)$.  Inductively, we form in this way a
  sequence of isotropic $G$-invariant subspaces $0 \subseteq V_1
  \subseteq V_2 \subseteq \cdots$ such that $\dim V_i = i$, and we
  terminate at $V_n$, since only for $i=n$ do we have $\dim V_i^\perp
  = i$.  Then, $G$ stabilizes the Lagrangian subspace $V_n$, and in
  the eigenbasis obtained from $v_1, \ldots, v_n$ together with their
  duals under the symplectic form, $G < (\C^{\times})^{n} < \GL_n <
  \Sp_{2n}$.

For the last statement, note that, if $G < (\CC^{\times})^n$, then in
standard symplectic coordinates, the elements $x_i y_i \in
\cO_{\CC^{2n}} = \CC[x_1,\ldots, x_n, y_1, \ldots, y_n]$ are
$G$-invariant.  Since, for a monomial $f$, $\{x_i y_i, f\} =
\deg_{x_i} f - \deg_{y_i} f$, it follows that $\HP_0(\cO_V^G, \cO_V)$
is a quotient, as a vector space, of the subalgebra $\CC[x_1 y_1, x_2
y_2, \ldots, x_n y_n] \subseteq \cO_{\CC^{2n}}$.  Since this
subalgebra is $G$-invariant, we deduce the statement.
\end{proof}

\begin{theorem}\label{theorem-typean-main}
$G<\CC{}^{\times{}}\times{}\CC{}^{\times{}}$ has the property $\HP_0(\cO_V^G,\cO_V)\cong{}\gr \HH_0(\cD_X^G,\cD_X)$ if and only if, up
to conjugation, $G$ is one of the following groups (for $r, m, A, B \geq 1)$: 
\begin{enumerate}
\item{}
The cyclic group generated by
$\begin{pmatrix}
e^{2\pi{}i/m} & 0\\
0 & e^{\pm{}2r\pi{}i/m}
\end{pmatrix},$
where $\mathrm{gcd}(r,m)=1$, and either \\ $r | (m+1)$ or $r | (m-1)$.
\item{}
The cyclic group generated by
$\begin{pmatrix}
e^{\pm{}2\pi{}i/(mA)} & 0\\
0 & e^{2\pi{}i/m}
\end{pmatrix}.$
\item{}
The group generated by
$\begin{pmatrix}
e^{2\pi{}i/A} & 0\\
0 & 1
\end{pmatrix}$
and
$\begin{pmatrix}
1 & 0\\
0 & e^{2\pi{}i/B}
\end{pmatrix}$.
\end{enumerate}
\end{theorem}
The proof of the theorem yields a complete description of the
resulting graded vector space $\HP_0(\cO_V^G, \cO_V) \cong \gr
\HH_0(\cD_X^G, \cD_X)$. In particular, from Theorem
\ref{theorem-staircase} and Figures \ref{k1fig1} and \ref{k1fig2} (for
type (1)), \ref{f:t-abel-2} (for type (2)), and \ref{f:t-abel-3} (for
type (3)), we deduce
\begin{corollary}
In the three cases defined in Theorem \ref{theorem-typean-main} such that $\HP_0(\cO_V^G,\cO_V)\cong{}\gr \HH_0(\cD_X^G,\cD_X)$, 
\begin{enumerate}
\item{} Let us assume that $r \not \equiv \pm 1 \pmod m$; otherwise this case is covered in (2) below.  
Define $p,q \geq 1$ as in \S \ref{subsectioncase1}: namely, $1 < p,q < m/2$, $p \equiv \pm r \pmod m$, and $pq = m \pm 1$.  Without loss of generality (up to conjugating $G$ by the nontrivial permutation matrix) we can assume $p \leq q$. Then,
\begin{align*}
h(\HP_0(\cO_V^G, \cO_V);t) &= 1 + 2t^2 + 3t^4 + \dotsb{} + p t^{2p-2} + p t^{2p} + \dotsb{} + p t^{2q-2} + (p-1) t^{2q}\\
 &+ \dotsb{} + t^{2p+2q-4}, \text{ if } pq+1=m;\\
h(\HP_0(\cO_V^G, \cO_V);t) &= 1 + 2t^2 + 3t^4 + \dotsb{} + p t^{2p-2} + p t^{2p} + \dotsb{} + p t^{2q-2} + (p-1) t^{2q}\\
 &+ \dotsb{} + 3t^{2p+2q-8} + t^{2p+2q-6}, \text{ if } pq-1=m.
\end{align*}
\item{} In this case,
\begin{align*}
h(\HP_0(\cO_V^G, \cO_V);t) &= 1 + 2t^2 + \dotsb{} + (m-1) t^{2m-4} + (m-1) t^{2m-2} + \dotsb{} + (m-1) t^{2A-2}\\
 &+ (m-2) t^{2A} + \dotsb{} + t^{2m+2A-6}, \text{ if } m\leq{}A;\\
h(\HP_0(\cO_V^G, \cO_V);t) &= 1 + 2t^2 + \dotsb{} + A t^{2A-2} + A t^{2A} + \dotsb{} + A t^{2m-4} + (A-1) t^{2m-2}\\
 &+ \dotsb{} + t^{2m+2A-6}, \text{ if } m > A.
\end{align*}
\item{}
Without loss of generality, assume that $A\geq{}B$. Then 
\begin{align*}
h(\HP_0(\cO_V^G, \cO_V);t) &= 1 + 2t^2 + \dotsb{} + (B-1) t^{2B-4} + (B-1) t^{2B-2} + \dotsb{} + (B-1) t^{2A-4}\\
 &+ (B-2) t^{2A-2} + \dotsb{} + t^{2A+2B-8}.
\end{align*}
\end{enumerate}
\end{corollary}

The theorem will follow from a case-by-case analysis of the following
general combinatorial description of $\HP_0(\cO_V^G, \cO_V)$ for
arbitrary $G < \C^\times \times \C^\times < \GL_2 < \Sp_4$, which is
interesting in its own right.

Let $V_1$ be the minimal set of generators for the semigroup
$\{x_1^rx_2^s|x_1^rx_2^s\in{}\cO_V^G, (r,s) \neq (0,0)\}$ and $V_2$ be
the minimal set of generators for the semigroup
$\{x_1^ry_2^s|x_1^ry_2^s\in{}\cO_V^G, (r,s) \neq (0,0)\}$. Note that
the elements of $V_1$ are those $x_1^r x_2^s$ with $r, s \geq 0$ and
$(r,s) \neq (0,0)$ such that, for all other $x_1^{r'} x_2^{s'} \in
\cO_V^G$ with $r', s' \geq 0$, either $r < r'$ or $s < s'$, and
similarly for $V_2$.

  Construct a graph $\Gamma$ as follows. The vertices of $\Gamma$ are
  the points $(j,k)$ where $j,k \geq -1$. For each $(r,s)$ such that
  $x_1^rx_2^s\in{}V_1$, we draw an edge between $(a+r,b+s-1)$ and
  $(a+r-1,b+s)$ for every pair of nonnegative integers $a,b$; we then
  do the same for every $x_1^ry_2^s \in V_2$.  
\begin{definition}
  Let $\mathcal{C}$ be the set of connected components of $\Gamma$
  whose vertices are all pairs $(a,b)$ of nonnegative integers, and
  such that every pair of adjacent vertices comprises the endpoints of
  a unique edge.
\end{definition}
\begin{theorem}\label{theorem-staircase}  Pick for each $C \in \mathcal{C}$ a vertex $(a_C,b_C) \in C$. Then a basis of $\HP_0(\cO_V^G, \cO_V)$ is obtained by the image of the monomials $\{x_1^{a_C} x_2^{b_C} y_1^{a_C} y_2^{b_C} \mid C \in \mathcal{C}\}$.
\end{theorem}
\begin{corollary} The Hilbert series of $\HP_0(\cO_V^G, \cO_V)$ is $\sum_{C \in \mathcal{C}} t^{2a_C+2b_C}$. Its dimension is $|\mathcal{C}|$.
\end{corollary}
Let us describe the connected components of the theorem more
explicitly. Let $E := \{(r,s) \in \ZZ_{\geq 0}^2 \setminus \{(0,0)\}
\mid x_1^r x_2^s \in V_1 \text{ or } x_1^r y_2^s \in V_2\}$. Then,
a connected component $C$ of $\Gamma$ is in $\mathcal{C}$
if and only if it is one of the following:
\begin{enumerate}
\item A connected component which is a point $(a,b)$ with $a, b \geq
  0$, such that for all $(r,s) \in E$, either $a < r-1$ or $b < s-1$;
\item A connected component which is a chain $(a,b+c), (a+1, b+c-1),
  \ldots, (a+c,b)$ with $a,b, c \geq{}0$ such that there is exactly
  one edge between any two consecutive points in the
  chain. Equivalently, for any $0\leq{}i\leq{}c-1$, there is exactly
  one $(r,s)\in{}E$ such that $a+i\geq{}r-1$ and $b+c-i\geq{}s$.
\end{enumerate}
We will refer to connected components of the first type as ``points of
type (1)'' and connected components of the second type as ``chains of
type (2).'' Note that there may exist chains of type (2) consisting of
a single point.  We will not always make a distinction
between connected components consisting of a single point and the
point itself.

Note that elements of $E$ the form $(0,s)$ and $(r,0)$ may generate chains
$(a,b+c), (a+1,b+c-1), \ldots, (a+c,b)$ which satisfy all the conditions 
of type (2) except that either $a < 0$ or $b < 0$; these
are not included in $\mathcal{C}$.

In practice, to apply the above theorem, it is more convenient and
intuitive to draw a picture called the \emph{staircase}.  This is the
collection of vertices $(r-1, s-1)$ for $(r,s) \in E$, together with
some line segments as follows: Call a vertex $(r-1,s-1)$ a
\emph{corner} if $(r,s) \in E$ and, for all other $(r', s') \in E$,
either $r < r'$ or $s < s'$.  Note that the points
of type (1) above are exactly those $(a,b)$ such that, for
every corner $(r-1,s-1)$, either $a < r-1$ or $b < s-1$.  Order the
corners $(r_1, s_1), (r_2, s_2), \ldots$ such that $r_1 < r_2 <
\cdots$.  We then draw line segments from $(r_i, s_i)$ to $(r_{i+1},
s_i)$ and from $(r_{i+1}, s_i)$ to $(r_{i+1}, s_{i+1})$.  Let the
\emph{staircase} be the region
\[
S := \{(x,y) \in \RR_{\geq 0}^2 \mid x \leq r_i-1 \text{ or } y \leq
s_i-1, \forall i\}.
\]
In general, this region is shaped like a staircase, which explains our
terminology. See Figures \ref{k1fig1}--\ref{f:t-abel-2} for examples
of the resulting staircases. In all of these figures except Figure
\ref{fig:staircase3}, the shaded regions consist only of vertices
lying in connected components in $\mathcal{C}$ (and every
connected component includes at least one vertex in the shaded region,
possibly on the boundary). Moreover, again in all figures except
Figure \ref{fig:staircase3}, the plotted vertices are exactly those
appearing in a connected component in $\mathcal{C}$.

Then, the points of type (1) are the lattice
points of $S$ which are not incident to any of the aforementioned line
segments (this includes all the lattice points in the interior of
$S$).  The chains of type (2) are naturally
in bijection with a subquotient of the remaining lattice points in
$S$, i.e., those incident to one of the aforementioned line segments.

\subsection{Proof of Theorem \ref{theorem-staircase}}
We begin with a series of preliminary lemmas.

\begin{lemma}
  $\cO_V^G$ is generated, as an algebra, by $x_1y_1$, $x_2y_2$, and
  the elements of the form $x_1^a x_2^b$, $x_1^ay_2^b$, $y_1^ax_2^b$,
  and $y_1^ay_2^b$.
\end{lemma}

\begin{proof}
It is clear that $x_1y_1$ and $x_2y_2$ are invariants. Since $G$ is a group of diagonal matrices, $f\in{}\cO_V$ is an invariant if and only if every term of $f$ is an invariant. For each monomial $x_1^{a_1}x_2^{a_2}y_1^{b_1}y_2^{b_2}$, if $a_1\geq{}b_1$ and $a_2\geq{}b_2$, then we can write $x_1^{a_1}x_2^{a_2}y_1^{b_1}y_2^{b_2}=(x_1y_1)^{b_1}(x_2y_2)^{b_2}(x_1^{a_1-b_1}x_2^{a_2-b_2})$. The other cases are similar.
\end{proof}

\begin{lemma} \label{l:aeqb}
If $a_1\neq{}b_1$ or $a_2\neq{}b_2$, then $x_1^{a_1}x_2^{a_2}y_1^{b_1}y_2^{b_2}\in{}\{\cO_V^G,\cO_V\}$.
\end{lemma}
\begin{proof}
  This is a special case of the argument of the proof of the final
  statement of Lemma \ref{l:abelian}. Explicitly, if $a_1\neq{}b_1$,
  then
\begin{align*}
\frac{1}{a_1-b_1}\{x_1y_1,x_1^{a_1}x_2^{a_2}y_1^{b_1}y_2^{b_2}\}&=x_1^{a_1}x_2^{a_2}y_1^{b_1}y_2^{b_2}.
\end{align*}
If $a_2\neq{}b_2$, then
\begin{align*}
\frac{1}{a_2-b_2}\{x_2y_2,x_1^{a_1}x_2^{a_2}y_1^{b_1}y_2^{b_2}\}&=x_1^{a_1}x_2^{a_2}y_1^{b_1}y_2^{b_2}. \qedhere
\end{align*}
\end{proof}
\begin{proof}[Proof of Theorem \ref{theorem-staircase}]
  By the above lemmas and Lemma \ref{l:pg}, it suffices to determine, for all $a,b \geq 0$,
  whether or not $x_1^ax_2^by_1^ay_2^b \in \{\cO_V^G, \cO_V\}$. By
  symmetry, $\{y_1^ry_2^s|x_1^rx_2^s\in{}V_1\}$ is a minimal set of
  generators of the semigroup of invariants of the form $y_1^ry_2^s$, and
  $\{y_1^rx_2^s|x_1^ry_2^s\in{}V_2\}$ is a minimal set of generators
  of the semigroup of invariants of the form $y_1^rx_2^s$. Furthermore,
\begin{align*}
\{x_1^rx_2^s,\cO_V\}\cap{}\{x_1^ax_2^by_1^ay_2^b \mid a,b \geq 0\}&=\{y_1^ry_2^s,\cO_V\}\cap{}\{x_1^ax_2^by_1^ay_2^b \mid a,b \geq 0\},\\
\{x_1^ry_2^s,\cO_V\}\cap{}\{x_1^ax_2^by_1^ay_2^b \mid a,b \geq 0\}&=\{y_1^rx_2^s,\cO_V\}\cap{}\{x_1^ax_2^by_1^ay_2^b \mid a,b \geq 0\}.
\end{align*}
So, $\{\cO_V^G,\cO_V\}$ is spanned by
$\{V_1,\cO_V\}$ and $\{V_2,\cO_V\}$, together with $\{x_1^a
y_1^b x_2^c y_2^d \mid (a,b) \neq (c,d)\}$.  Next,
\begin{align*}
\{x_1^{r}x_2^{s},x_1^{a_1}x_2^{a_2}y_1^{b_1}y_2^{b_2}\}&=sb_2x_1^{a_1+r}x_2^{a_2+s-1}y_1^{b_1}y_2^{b_2-1}+rb_1x_1^{a_1+r-1}x_2^{a_2+s}y_1^{b_1-1}y_2^{b_2},\\
\{x_1^{r}y_2^{s},x_1^{a_1}x_2^{a_2}y_1^{b_1}y_2^{b_2}\}&= -sa_2x_1^{a_1+r}x_2^{a_1-1}y_1^{b_1}y_2^{b_2+s-1}+rb_1x_1^{a_1+r-1}x_2^{a_1}y_1^{b_1-1}y_2^{b_2+s}.
\end{align*}
We are interested in the possible RHS expressions whose
monomials have the form $x_1^{a}x_2^{b}y_1^{a}y_2^{b}$:
\begin{align*}
\{x_1^{r}x_2^{s},x_1^{a_1}x_2^{a_2}y_1^{a_1+r}y_2^{a_2+s}\}&=s(a_2+s)x_1^{a_1+r}x_2^{a_2+s-1}y_1^{a_1+r}y_2^{a_2+s-1}+r(a_1+r)x_1^{a_1+r-1}x_2^{a_2+s}y_1^{a_1+r-1}y_2^{a_2+s},\\
\{x_1^{r}y_2^{s},x_1^{a_1}x_2^{a_2+s}y_1^{a_1+r}y_2^{a_2}\}&= -s(a_2+s)x_1^{a_1+r}x_2^{a_2+s-1}y_1^{a_1+r}y_2^{a_2+s-1}+r(a_1+r)x_1^{a_1+r-1}x_2^{a_2+s}y_1^{a_1+r-1}y_2^{a_2+s}.
\end{align*}
For simplicity, denote
$[f]=f+\{\cO_V^G,\cO_V\}\in{}\HP_0(\cO_V^G,
\cO_V)$. Then, for every $x_1^r x_2^s \in V_1$,
\begin{equation}\label{relation1}
  s_i(a_2+s_i)[x_1^{a_1+r_i}x_2^{a_2+s_i-1}y_1^{a_1+r_i}y_2^{a_2+s_i-1}] +
r_i(a_1+r_i)[x_1^{a_1+r_i-1}x_2^{a_2+s_i}y_1^{a_1+r_i-1}y_2^{a_2+s_i}]=0.
\end{equation}
For every $x_1^r y_2^s \in V_2$,
\begin{equation}\label{relation2}
  -s_i(a_2+s_i)[x_1^{a_1+r_i}x_2^{a_2+s_i-1}y_1^{a_1+r_i}y_2^{a_2+s_i-1}]+
  r_i(a_1+r_i)[x_1^{a_1+r_i-1}x_2^{a_2+s_i}y_1^{a_1+r_i-1}y_2^{a_2+s_i}]=0,
\end{equation}
if $r,s \geq 1$; in the case that $s=0$, 
\begin{equation}\label{relation3}
[x_1^{a_1+r-1}x_2^{a_2}y_1^{a_1+r-1}y_2^{a_2}]=0,
\end{equation}
and in the case that $r=0$, 
\begin{equation}\label{relation4}
[x_1^{a_1}x_2^{a_2+s-1}y_1^{a_1}y_2^{a_2+s-1}]=0.
\end{equation}
Since $V_1 \cup V_2$ forms a set of algebra generators of $\cO_V^G$,
these span all the relations in $\HP_0(\cO_V^G,\cO_V)$, together with
the relations $[x_1^a x_2^b y_1^c y_2^d] = 0$ if $a \neq c$ or $b \neq
d$. Now, if we represent $[x_1^{a_1}x_2^{a_2}y_1^{a_1}y_2^{a_2}]$ by
the point $(a_1,a_2)$ and each relation by an edge, then we get the
subgraph of $\Gamma$ of vertices with nonnegative coordinates,
together with the additional relations that $[x_1^{a_1}
x_2^{a_2}y_1^{a_1}y_2^{a_2}] =0$ if $(a_1, a_2)$ is adjacent in
$\Gamma$ to a vertex that does not have nonnegative coordinates.

Let $C_1,C_2,\ldots$ be the connected components of $\Gamma$
containing at least one vertex with nonnegative coordinates. Let $V(C_i)
\subseteq\HP_0(\cO_V^G,\cO_V) $ be the vector space spanned
by $\{[x_1^rx_2^sy_1^ry_2^s]|(r,s)\in{}C_i,r,s\geq{}0\}$. Then
$\HP_0(\cO_V^G,\cO_V) = \bigoplus_i V(C_i)$.

For any $a,b\geq{}0$, if for every $(r,s) \in E$, either $a<r-1$ or
$b<s-1$, then there is no relation involving
$[x_1^ax_2^by_1^ay_2^b]$. Thus, $\dim{}V(\{(a,b)\})=1$. This accounts
for the points of type (1). Next, if $a', b' \geq 0$ and there exists
$(r,s) \in E$ such that $a'\geq r-1$ and $b'\geq s-1$, then $(a',b')$
is in a connected component of $\Gamma$ that is a chain of the form
$(a,b+c),(a+1,b+c-1), \ldots, (a+c,b)$. If there is exactly one edge
between any two consecutive points $(a+i,b+c-i)$ and
$(a+i+1,b+c-i-1)$, and $a, b \geq 0$, then there is exactly one
relation of the form \ref{relation1} or \ref{relation2} between the
two corresponding terms $[x_1^{a+i}y_1^{b+c-i}x_2^{a+i}y_2^{b+c-i}]$
and $[x_1^{a+i+1}y_1^{b+c-i-1}x_2^{a+i+1}y_2^{b+c-i-1}]$, and no other
relations involving these elements. Therefore,
$\dim{}V(\{(a,b+c),(a+1,b+c-1),\dotsc{},(a+c,b)\})=1$. This accounts
for the chains of type (2).

If there are two edges between two consecutive points of a chain, then
there are two relations of the form \ref{relation1} or
\ref{relation2}. The assumption that $V_1$, $V_2$ are minimal sets of
generators implies that the two relations are irredundant. Therefore,
$V(\{(a,b+c),(a+1,b+c-1),\dotsc{},(a+c,b)\})=0$. Finally, if a
connected component $C_i$ contains a point $(a,b)$ with $a=-1$ or
$b=-1$, then there is a relation of the form \ref{relation3} or
\ref{relation4}, which implies that $V(C_i)=0$.
\end{proof}

\subsection{Proof of Theorem \ref{theorem-typean-main}}
We prove Theorem \ref{theorem-typean-main} first in the case that $G$
is cyclic and generated by an element of the form
\begin{equation}\label{case1}
\begin{pmatrix}
e^{2\pi{}i/m} & 0\\
0 & e^{2r\pi{}i/m}
\end{pmatrix},
\end{equation} where $\mathrm{gcd}(r,m)=1$ (\textbf{Case I}), and then we
reduce the general case (\textbf{Case II}) to this case.

\subsubsection{Case I: $G$ is generated by \eqref{case1}} \label{subsectioncase1}
In this subsection, we prove the most difficult part of the theorem:
\begin{proposition} \label{p:theorem-typean-main-case1}
Let $G$ be cyclic and generated by
$\begin{pmatrix}
e^{2\pi{}i/m} & 0\\
0 & e^{2r\pi{}i/m}
\end{pmatrix}$ where $\mathrm{gcd}(r,m)=1$. Assume
$|r|\leq{}\frac{m}{2}$. Then, $G$ has the property
$\HP_0(\cO_V^G,\cO_V)\cong \gr
\HH_0(\cD_X^G,\cD_X)$ if and only if $r|(m+1)$ or
$r|(m-1)$.
\end{proposition}
Since $\mathrm{gcd}(r,m)=1$, it follows from Lemma \ref{l:afls-fla},
as mentioned at the beginning of the section, that $\dim
\HH_0(\cD_X^G, \cD_X) = |G|-1$.

We break the proof into two easy lemmas and one hard one.

Since $G$ is generated by 
$\begin{pmatrix}
  e^{2\pi{}i/m} & 0\\
  0 & e^{2r\pi{}i/m}
\end{pmatrix}$, it follows in the case $r > 0$ that
$x_1^{r} y_2$ is an invariant, and in the case $r < 0$ that $x_1^{-r}
x_2$ is an invariant.  Since also $|r| \leq m/2$, $(|r|-1,0)$ is a corner of the
staircase. Next, let $t$ be an integer such that $|t| \leq m/2$ and
$rt \equiv 1 \pmod m$.  Then, $G$ is also generated by $\begin{pmatrix}
  e^{2t\pi{}i/m} & 0\\
  0 & e^{2\pi{}i/m}
\end{pmatrix}$.
It follows that $(0,|t|-1)$ is a corner of the staircase.  For ease of notation, let us set $p := |r|$ and $q := |t|$, so that $(p-1,0)$ and $(0,q-1)$ are
corners of the staircase.

Since $rt\equiv 1 \pmod{m}$, it follows that either $m| (pq+1)$ or $m|
(pq-1)$.  It suffices to assume that $G$ is nontrivial, i.e., $m > 1$.
Let $k \geq 0$ be such that $mk=pq+1$ or $mk=pq-1$. Then the
proposition reduces to the following lemmas:
\begin{lemma}
If $k=0$, then $\dim{}\HP_0(\cO_V^G,\cO_V)=\dim{}\HH_0(\cD_X^G,\cD_X)$.
\end{lemma}
\begin{proof} 
  In this case, $p=q=1$.  Then, $(0,0)$ is a corner of the staircase, as are
  $(m-1,0)$ and $(0,m-1)$. The proposition follows easily.
\end{proof}
\begin{lemma}
If $k=1$, then $\dim{}\HP_0(\cO_V^G,\cO_V)=\dim{}\HH_0(\cD_X^G,\cD_X)$.
\end{lemma}
\begin{proof}
  If $k=1$, then $m=pq+1$ or $pq-1$. It is straightforward to
  compute $\dim{}\HP_0(\cO_V^G,\cO_V)$ from Figures
  \ref{k1fig1} and \ref{k1fig2}, which depict the corresponding staircases.
\end{proof}
\begin{figure}[h!tbp] 
\centering
	\includegraphics[width = 0.5\textwidth]{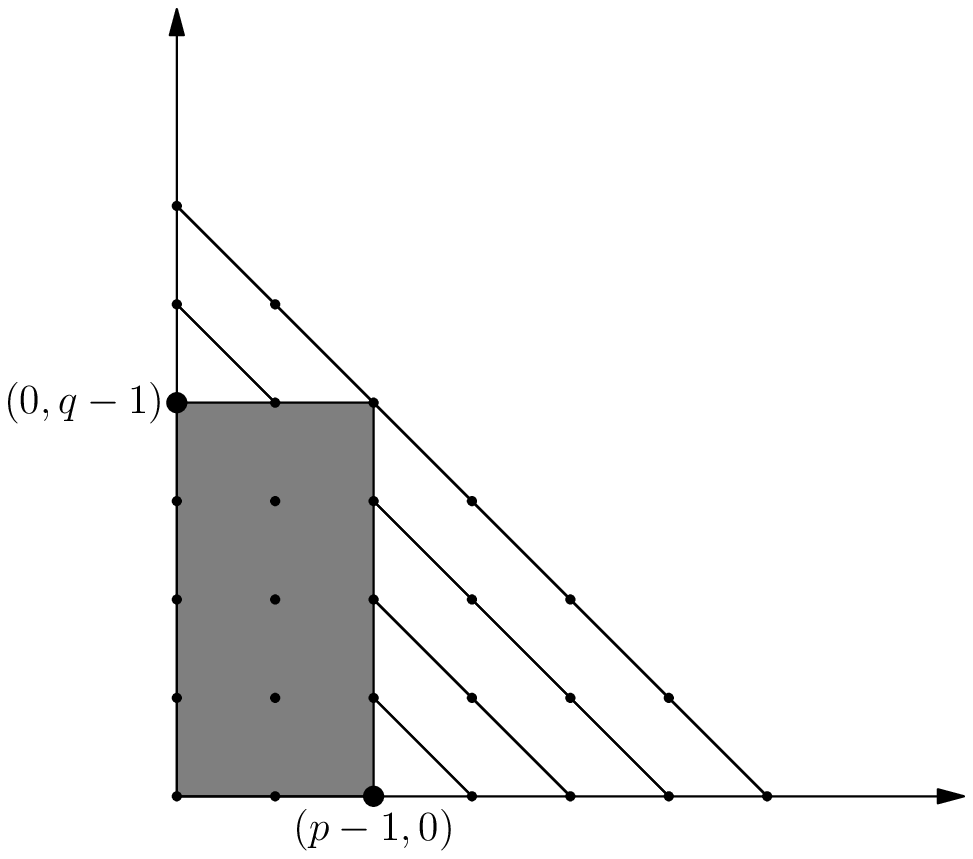}
	\caption{The case m=pq+1\label{k1fig1} }
\end{figure}
\begin{figure} [h!tbp] 
\centering
\includegraphics[width = 0.5\textwidth]{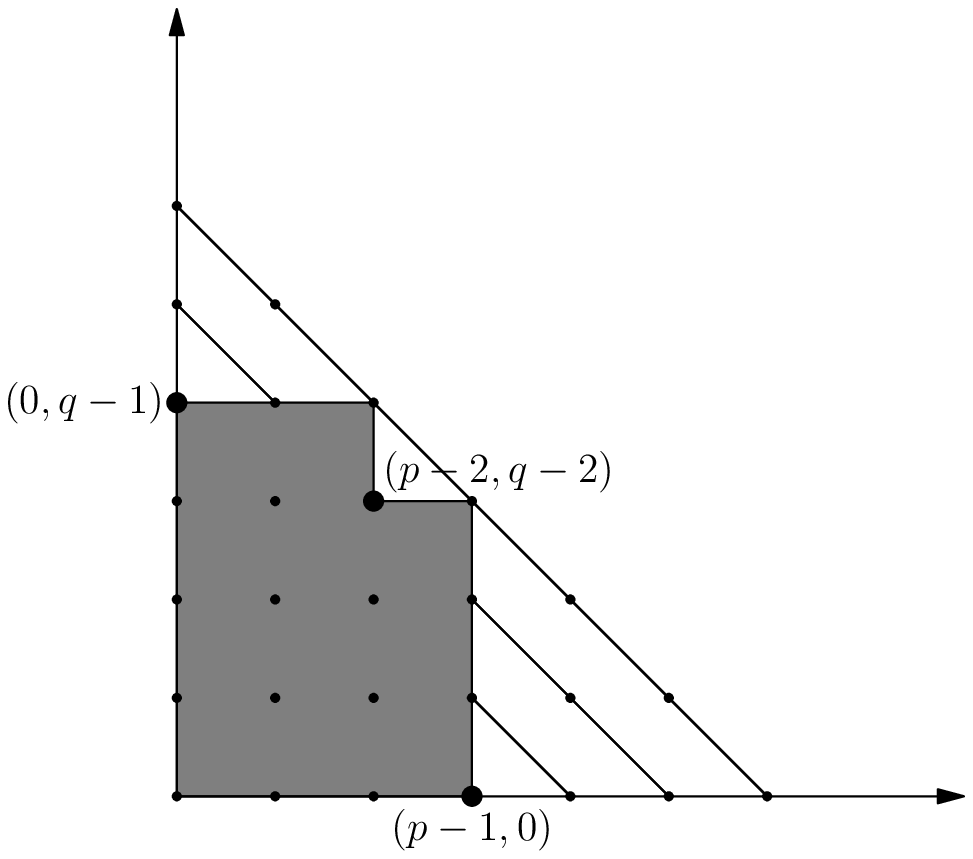}
	\caption{The case m=pq-1\label{k1fig2}}
\end{figure}
\begin{lemma}
  If $k\geq{}2$ and $m > 1$, then
  $\dim{}\HP_0(\cO_V^G,\cO_V) >
  \dim{}\HH_0(\cD_X^G,\cD_X)$.
\end{lemma}
The proof of this final lemma is long and somewhat technical, so we
further subdivide it into several parts.
\begin{proof}
  Note that, by assumption, $p, q > 1$. Write $m=bp+a$ for $0<a<p$ and
  $m=cq+d$ for $0<d<q$. 
\begin{claim}\label{c:abcd}
$(a-1,b-1)$ and $(c-1,d-1)$ are corners of the staircase: $(a-1,b-1)$ is the rightmost before $(p-1,0)$, and $(c-1,d-1)$ is the leftmost after $(0,q-1)$, as
in Figure \ref{fig:staircase3}.
\end{claim}
\begin{proof}
  First, note that $b < q/k$ and $c < p/k$, since $m = \frac{pq \pm 1}{k}
  = bp+a = cq+d$. Next, for all $a'$ such that $a < a' < p$, $a' + b'
  p \equiv 0 \pmod m$ implies that $b'p > m$ so that $b' > q/k$.
  Therefore,
  $(a'-1,b'-1)$ cannot be a corner of the staircase.  It follows that
  $(a-1,b-1)$ is a corner.  Similarly, if $d < d' < q$, then
  $(c'-1,d'-1)$ cannot be a corner, and hence $(c-1,d-1)$ is a corner.
\end{proof}
In particular, it follows that $c \leq a$ and $d \geq b$ (see Figure
\ref{fig:staircase3}).  (A direct proof of this also follows from the
argument of the proposition: first one shows $c < p/k$ and $b < q/k$; then
if $a < c < p$, it would follow that $d > q/k$, a contradiction.) To summarize,
$0 < c \leq a < p$ and $0 < b \leq d < q$.

Note also that
$b=\lfloor{}\frac{m}{p}\rfloor{}=\lfloor{}\frac{q}{k}\rfloor{}$ and
$c=\lfloor{}\frac{m}{q}\rfloor{}=\lfloor{}\frac{p}{k}\rfloor{}$. By
our assumptions, $p,q < m/2$, and hence also $b,c \geq 2$.

\begin{claim}
$p+b-2\leq{}m-p$.
\end{claim}

\begin{proof}
\begin{align*}
(m-p)-(p+b-2)&=m+2-2p-b\\
&\geq{}m+2-2p-\frac{m}{p}.
\end{align*}
Let $f(p)=m+2-2p-\frac{m}{p}$. Since $f(p)$ is convex and $1 <p <
\frac{m}{2}$, it suffices to prove that $f(1)\geq 0$ and
$f(\frac{m}{2})\geq 0$. This is clear because they are both $0$.
\end{proof}

Therefore, glancing at Figure \ref{fig:staircase3}, we see that there
are chains beginning at $(p-1,0), \ldots, (p-1,b-2)$ of type (2) (in
the language of the beginning of the section) which form connected
components in $\mathcal{C}$.  Similarly, there are chains of type
(2) ending at
$(0,q-1),\dotsc{},(c-2,q-1)$.
  Next, again from Figure \ref{fig:staircase3}, we see
that there are points of type (1) of the form $(c-1,j)$ with $b-1
\leq j < d-1$ and of the form $(i, b-1)$ for $c-1 \leq i < a-1$, and
also the chains $(c-1, d-1)$ and $(a-1, b-1)$ of type (2), each a
connected component in $\mathcal{C}$ consisting of a single vertex
(some of which may be equal).  Together with the more obvious points
$(i,j)$ of type (1) where either $i < c-1, j < q-1$ or $i < p-1, j <
b-1$, we deduce
\begin{claim}
$\dim{}\HP_0(\cO_V^G,\cO_V)\geq{}p(b-1)+q(c-1)-
(b-1)(c-1)+(d-b)+(a-c)+1$.
\end{claim}
Let $h$ denote the difference $h := \dim{}\HP_0(\cO_V^G,\cO_V) -
\bigl( p(b-1)+q(c-1)- (b-1)(c-1)+(d-b)+(a-c)+1 \bigr)$.  In
particular, $h$ is at least the number of chains of type (2)
containing vertices $(j,k)$ such that $j+k > \max\{c+d-2,a+b-2\}$ and
$j < p-1, k < q-1$.  (The last condition ensures that these chains are
not the ones beginning with any of the vertices $(p-1,0), \ldots,
(p-1, b-2)$ or ending at any of the vertices $(0,q-1), \ldots, (c-2,
q-1)$, which we already counted above.)

\begin{figure} [h!tbp] 
\centering
	\includegraphics{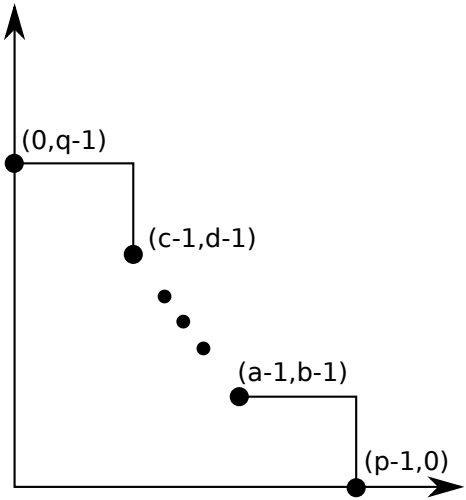}
	\caption{The staircase for $k \geq 2$\label{fig:staircase3}}
\end{figure}

In view of the claim and the formula for
$\dim{}\HH_0(\cD_X^G,\cD_X)$, we deduce that 
\begin{align*}
  \dim{}\HP_0&(\cO_V^G,\cO_V)-
  \dim{}\HH_0(\cD_X^G,\cD_X)\\
  &= p(b-1)+q(c-1)-(b-1)(c-1)+(d-b)+(a-c)+1-(m-1) + h\\
  &=m-a-p+m-d-q-bc+b+c-1+d-b+a-c+1-m+1 + h\\
  &=m+1-p-q-bc + h.
\end{align*}

We will need one more inequality which gives a lower bound on $p$, and
similarly on $q$.

\begin{claim}
$p\geq{}kc+1$. Similarly, $q \geq kb+1$.
\end{claim}

\begin{proof}
$pq\geq{}km-1=k(cq+d)-1>kcq$.  The same argument shows $q > kb$.
\end{proof}

We now divide the lemma into five cases. In each case, we 
prove that $m+1-p-q-bc + h > 0$.  Up to symmetry (swapping $r$ with
$t$), we will assume that $b \geq c$.

{\bf Case 1.} $k=2$. Note that, since $b, c \geq 2$ as remarked at the
beginning of the proof of the lemma, it follows that $p \geq kc+1 \geq
5$ and similarly $q \geq 5$.

{\bf Case 1a.} $m=\frac{pq-1}{2}$.

In this case, the staircase has three corners with nonnegative
coefficients: $(p-1, 0)$,
$(\frac{p-1}{2}-1,\frac{q-1}{2}-1)$, and $(0,q-1)$. So
$a=c=\frac{p-1}{2}$ and $b=d=\frac{q-1}{2}$. Then,
\begin{align*}
m+1-p-q-bc&=\frac{pq-1}{2}+1-p-q-\frac{p-1}{2} \cdot \frac{q-1}{2}\\
&=\frac{1}{4}(pq-3p-3q+1).
\end{align*}
In addition, since $p,q \geq 5$, we have at least two additional
chains in $\mathcal{C}$ of type (2):
$(\frac{p-3}{2},\frac{q-1}{2}),(\frac{p-1}{2},\frac{q-3}{2})$ and
$(\frac{p-3}{2},\frac{q+1}{2}),(\frac{p-1}{2},\frac{q-1}{2}),
(\frac{p+1}{2},\frac{q-3}{2})$. So $h \geq 2$, and it suffices to
prove that $pq-3p-3q+9 = (p-3)(q-3) >0$, which is obvious.

{\bf Case 1b.} $m=\frac{pq+1}{2}$.

In this case, the staircase has four corners with nonnegative
coefficients: $(0, q-1)$, $(\frac{p-1}{2}-1,\frac{q+1}{2}-1)$,
$(\frac{p+1}{2}-1,\frac{q-1}{2}-1)$, and $(0,p-1)$. So $a=\frac{p+1}{2}$,
$b=\frac{q-1}{2}$, $c=\frac{p-1}{2}$, $d=\frac{q+1}{2}$. Then,
\begin{align*}
m+1-p-q-bc&=\frac{pq+1}{2}+1-p-q-\frac{p-1}{2} \cdot \frac{q-1}{2}\\
&=\frac{1}{4}(pq-3p-3q+5).
\end{align*}
Also, since $p,q \geq 5$, there is at least one additional chain of
type (2) in $\mathcal{C}$:
$(\frac{p-3}{2},\frac{q+1}{2}),(\frac{p-1}{2},\frac{q-1}{2}),
(\frac{p+1}{2},\frac{q-3}{2})$. So
$h \geq 1$, and it suffices to prove that $pq-3p-3q+9>0$, which we already
saw in Case 1a.

{\bf Case 2.} $k\geq{}3$, $b\geq{}3$, $c\geq{}3$.

In this case, $m+1-p-q-bc>0$ follows from the inequalities
\begin{align*}
  p&<\frac{m}{b}\leq{}\frac{m}{3},\\
  q&<\frac{m}{c}\leq{}\frac{m}{3}, \text{ and }\\
  bc&=\frac{m-a}{p} \cdot
  \frac{m-d}{q}<\frac{m^2}{pq}\leq{}\frac{m\frac{pq+1}{k}}{pq}=
\frac{m}{k}+\frac{m}{kpq}<\frac{m}{3}+1.
\end{align*}

{\bf Case 3.} $k\geq{}3$, $c=2$, $b\geq{}4$.

Since $p\geq{}kc+1\geq{}7$, 
\begin{align*}
p&<\frac{m}{b}\leq{}\frac{m}{4},\\
q+\frac{1}{2}b&\leq{}q+\frac{d}{2}=\frac{m}{2}, \text{ and }\\
\frac{3}{2}b&<\frac{3m}{2p}\leq{}\frac{3m}{14}.
\end{align*}
For the inequality on the second line, see Figure \ref{fig:staircase3}
and the discussion after Claim \ref{c:abcd}. We deduce from the three
lines that $m+1-p-q-bc=m+1-p-(q+\frac{1}{2}b)-(\frac{3}{2}b)>0$.

{\bf Case 4.} $k\geq{}3$, $c=2$, $b=3$.  

Note that $d \geq b = 3$ and $a \geq c = 2$. Hence
\begin{align*}
q&=\frac{m-d}{c}\leq{}\frac{m-3}{2} \text{ and }\\
p&=\frac{m-a}{b}\leq{}\frac{m-2}{3}.
\end{align*}
So, $m-p-q-5\geq{}\frac{m-17}{6}$. Since $m>bp>bkc\geq{}18$, we conclude
that $m-p-q-5>0$, as desired.

{\bf Case 5.} $k\geq{}3$, $c=2$, $b=2$.  Note that

\begin{align*}
m+1-p-q-bc&=m+1-\frac{m-a}{2}-\frac{m-d}{2}-4\\
&=\frac{a+d-6}{2}
\end{align*}

Therefore, it suffices to prove that $2h + a+d > 6$. 

{\bf Case 5a.} $a=d=2$.

In this case, we have at least two additional chains of type (2)
in $\mathcal{C}$:
$(1,2),(2,1)$ and $(1,3),(2,2),(3,1)$. Therefore, $h \geq 2$, as desired.

{\bf Case 5b.} If we are not in the case $a=d=2$, then $(1,1)$ is not
a corner of the staircase; in view of Figure \ref{fig:staircase3},
this implies  $a,d>2$. It suffices to assume that $a=d=3$. We claim
that this cannot happen. For sake of contradiction, assume
$a=d=3$. Then, $m=2p+3=2q+3$. Since $m=\frac{pq\pm 1}{k}$, $4m$
divides $4(pq \pm 1)=m^2-6m+9\pm 4$. Therefore, $m$ is odd, so $m \mid
m^2 - 6m +9 \pm 4$, and hence $m$ divides $5$ or $13$. However,
$m=2p+3\geq{}2(kc+1)+3 \geq 17$, which is a contradiction.
\end{proof}

\subsubsection{Case II: the general case}

In this subsection, we complete the proof of Theorem
\ref{theorem-typean-main} by reducing the general case to Proposition
\ref{p:theorem-typean-main-case1}, which was proved in the previous
subsection.

\begin{lemma}
Let $A=\min \{r > 0: x_1^r x_2^s \in \cO_V^G \text{ or } x_1^r y_2^s \in \cO_V^G\}$. Then, for every invariant of the form $x_1^r x_2^s$ or $x_1^r y_2^s$
in $\cO_V^G$, 
$A \mid r$.
\end{lemma}

\begin{proof}
  It is enough to show the result for $r > 0$.  Suppose, for sake of
  contradiction, that $A \nmid r$ and $x_1^r x_2^s$ or $x_1^r y_2^s$
  is an invariant.  We can assume that $r$ is minimal for this
  property.  There must exist $s', s'' \geq 0$ such that $x_1^A
  x_2^{s'}$ and $x_1^A y_2^{s''}$ are invariants.  In the first case
  that $x_1^r x_2^s$ is invariant, it follows also that
  $x_1^{r-A}x_2^{s+s''}$ is invariant; in the latter case that $x_1^r
  y_2^s$ is invariant, it follows also that $x_1^{r-A} y_2^{s+s'}$ is
  invariant. This contradicts our assumption.
\end{proof}

Similarly, let $B= \min \{s > 0: x_1^r x_2^s \in \cO_V^G \text{ or }
x_1^r y_2^s \in \cO_V^G\}$. Then $B$ divides all of the $s$ appearing
in the set.  We construct a group $G'$ in the following way:
\begin{equation*}
  G'=\left\{\begin{pmatrix}\zeta^A & 0 \\ 0 & \xi^B\end{pmatrix}: 
\begin{pmatrix}\zeta & 0 \\ 0 & \xi \end{pmatrix} \in{}G\right\}.
\end{equation*}
Then $x_1^{Ar}x_2^{Bs}$ is an invariant of $G$ if and only if
$x_1^rx_2^s$ is an invariant of $G'$, and $x_1^{Ar}y_2^{Bs}$ is an
invariant of $G$ if and only if $x_1^ry_2^s$ is an invariant of $G'$.

\begin{lemma} \label{l:GfromG'}
$G=\left\{\begin{pmatrix}\zeta & 0 \\ 0 & \xi\end{pmatrix} : \begin{pmatrix}\zeta^A & 0 \\ 0 & \xi^B \end{pmatrix} \in{}G'\right\}$.
\end{lemma}

\begin{proof}
  It is immediate from the above description that the two groups have
  the same invariants. This implies that the two groups are the same
  in a standard way: for example, if $G \leq H$ and $\cO_V^G =
  \cO_V^H$, then the quotient fields $\CC(V)^G$ and $\CC(V)^H$ would also
  be equal, and by the main theorem of Galois theory, $G = H$.
\end{proof}

\begin{lemma}
$G'$ is generated by $\begin{pmatrix}
e^{2\pi{}i/m} & 0\\
0 & e^{2r\pi{}i/m}
\end{pmatrix}$, for some integers $r, m$ with $\mathrm{gcd}(r,m)=1$.
\end{lemma}

\begin{proof}
  Let $m \geq 1$ be the positive integer such that the first
  projection $\{\zeta: \begin{pmatrix}\zeta & 0 \\ 0 &
    \xi\end{pmatrix} \in G' \}$ is the cyclic group generated by
  $e^{2\pi i/m}$. By the definition of $G'$, there exists $\ell \geq
  1$ such that $x_1^\ell x_2 \in \cO_V^{G'}$.  It follows that the
  lattice $(\ZZ^2)^{G'} := \{(a,b) \in \ZZ^2 \mid x_1^a x_2^b \in
  \CC(V)^{G'}\}$ is generated by $(m,0)$ and $(\ell, 1)$. By
  assumption, $\mathrm{gcd}(\ell,m) = 1$.  Thus, we can let $r := -\ell$,
  and then $(\ZZ^2)^{G'}$ identifies with the lattice invariant under
  the element stated in the lemma.  This implies that $G'$ is
  generated by the element.  In more detail, if $K \leq G'$ is the
  subgroup generated by this element, then $|K| = |\ZZ^2/(\ZZ^2)^K| =
  |\ZZ^2/(\ZZ^2)^{G'}|= |G'|$.
\end{proof}

We see that Case I of Theorem \ref{theorem-typean-main}, i.e.,
Proposition \ref{p:theorem-typean-main-case1}, is equivalent to the
case $A=B=1$.  We divide the remainder of the theorem into two cases:

{\bf Case 1.} $A,B>1$.  In the case that $G'$ is the trivial group,
$G$ is evidently of the type (3) in Theorem
\ref{theorem-typean-main}, and it is easy to see that, for this group,
$\dim \HP_0(\cO_V^G, \cO_V) = (A-1)(B-1) = \dim \HH_0(\cD_X^G,
\cD_X)$. See also Figure \ref{f:t-abel-3}.
\begin{claim}
If $A,B>1$, and $G'$ is nontrivial, then $\dim{}\HP_0(\cO_V^G,\cO_V)>\dim{}\HH_0(\cD_X^G,\cD_X)$.
\end{claim}

\begin{proof}
Without loss of generality, assume that $A\geq{}B$. Then $\dim{}\HH_0(\cD_X^G,\cD_X)=ABm-A-B+1$. Then, we prove that $\dim{}\HP_0(\cO_V^G,\cO_V)\geq{}AB\dim{}\HP_0(\cO_V^{G'},\cO_V)+2(A-1)(B-1)$ by the following correspondence:

(i) Let $(a,b)$ be a point that forms a connected component of $\Gamma(G')$ of type (1). Then, for every $(r,s) \in E(G)$, either
$a<r/A-1$ or $b<s/B-1$. Hence, $(Aa+i,Bb+j)$ forms a connected component of $\Gamma(G)$ of type (1) for each $0\leq{}i<A$, $0\leq{}j<B$, because $Aa+i<r-1$ or $bB+j<s-1$ for all $(r,s) \in E(G)$.

(ii) Let $(a,b+c),(a+1,b+c-1), \ldots, (a+c,b)$ form a connected
component of $\Gamma(G')$ of type (2).  Then we can verify that
$(Aa+i,B(b+c)+j)$ is a connected component of $\Gamma(G)$ of type (1)
for each $0\leq{}i < A-1$, $0\leq{}j<B$, and that the chains
starting from $(Aa+A-1,B(b+c)+j)$, $0\leq{}j<B$ are connected
components of $\Gamma(G)$ of type (2).

(iii) In addition, each point $(A(m-1)+i,j)$ and $(i,B(m-1)+j)$,
$0\leq{}i<A-1$, $0\leq{}j<B-1$ forms a connected component of
$\Gamma(G)$ of type (1).

Thus,
$\dim{}\HP_0(\cO_V^G,\cO_V)\geq{}AB\dim{}\HP_0(\cO_V^{G'},\cO_V)+2(A-1)(B-1)\geq{}AB(m-1)+2(A-1)(B-1)>ABm-A-B+1=\dim{}\HH_0(\cD_X^G,\cD_X)$.
\end{proof}

\begin{figure}[h!tbp] 
\centering
	\includegraphics[width = 0.5\textwidth]{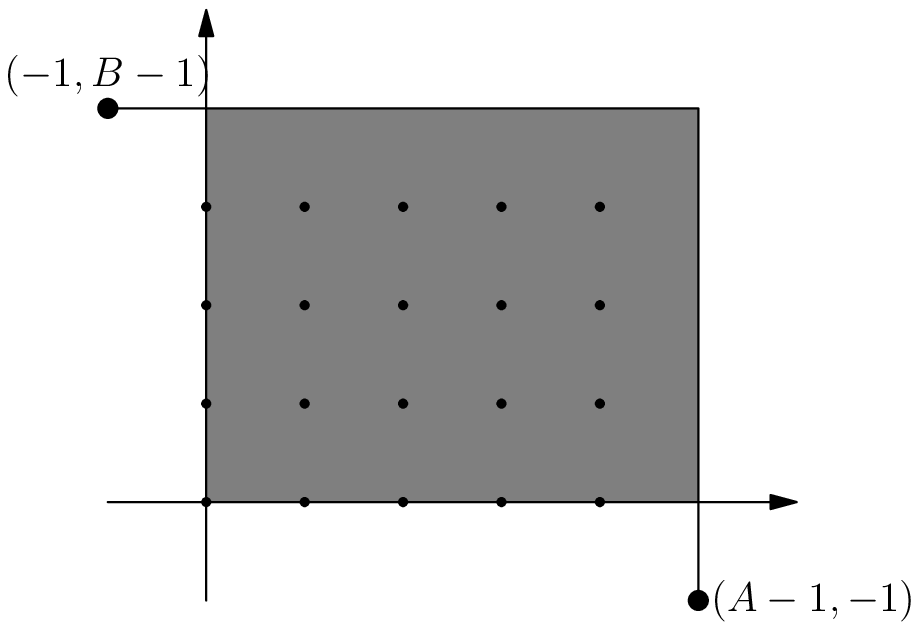}
	\caption{The staircase for type (3) in Theorem \ref{theorem-typean-main} \label{f:t-abel-3}}
\end{figure}

{\bf Case 2.} $A>1,B=1$ or $A=1,B>1$. Without loss of generality, assume that $A>1,B=1$.
\begin{claim}\label{c:CaseII2}
If $A > 1$ and $B=1$, $G$ is nontrivial, and $\dim{}\HP_0(\cO_V^G,\cO_V)=\dim{}\HH_0(\cD_X^G,\cD_X)$, then $G'$ is generated by
$\begin{pmatrix}
e^{2\pi{}i/m} & 0\\
0 & e^{\pm{}2\pi{}i/m}
\end{pmatrix}$.
\end{claim}
For $G'$ as in the claim, Lemma \ref{l:GfromG'}  implies that
$G$ is generated by
$\begin{pmatrix}
e^{\pm{}2\pi{}i/(mA)} & 0\\
0 & e^{2\pi{}i/m}
\end{pmatrix}$.  This accounts for the groups of type (2) in Theorem
\ref{theorem-typean-main}; conversely, it is an easy consequence of
Theorem \ref{theorem-staircase} that all of these groups indeed
satisfy
$\dim{}\HP_0(\cO_V^G,\cO_V)=\dim{}\HH_0(\cD_X^G,\cD_X)$. See
also Figure \ref{f:t-abel-2}. This
finishes the proof of the theorem, and it remains only to prove the
claim.
\begin{proof}[Proof of Claim \ref{c:CaseII2}]
  Similarly to (i) and (ii) in Case 1 above, 
\begin{gather*}
\dim{}\HH_0(\cD_X^G,\cD_X)=A(m-1)=A\dim{}\HH_0(\cD_X^{G'},\cD_X) \text{ and } \\
    \dim{}\HP_0(\cO_V^G,\cO_V)\geq{}A\dim{}\HP_0(\cO_V^{G'},\cO_V).
\end{gather*}
Assume that
$\dim{}\HP_0(\cO_V^G,\cO_V) =
\dim{}\HH_0(\cD_X^G,\cD_X)$. Then,
we must have
$\dim{}\HP_0(\cO_V^{G'},\cO_V) =
\dim{}\HH_0(\cD_X^{G'},\cD_X)$.

Define $p,q$ in the same way as in Case I (note that we must have $k =
0$ or $k=1$). Then $(0,q-1)$ is the corner of the staircase for $G'$
with $x$-coordinate equal to zero. This implies that the staircase for
$G$ has the corner $(A-1, q-1)$.  However, in this case, it would
follow, similarly to the argument in Case 1 of this subsection, that
$\dim \HP_0(\cO_V^G, \cO_V) > A \dim \HP_0(\cO_V^{G'}, \cO_V)$ unless
$q = 1$. In the latter case, $G'$ is as claimed.
\end{proof}

\begin{figure}[h!tbp] 
\centering
	\includegraphics[width = 0.5\textwidth]{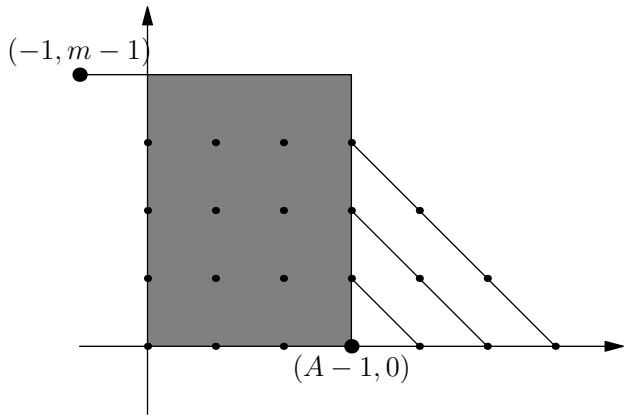}
	\caption{The staircase for type (2) in Theorem \ref{theorem-typean-main} \label{f:t-abel-2}}
\end{figure}

\section{Complex reflection groups $G(m,p,2) < \GL_2 < \Sp_4$} \label{s:gmp2}
Assume $m \geq 2$.
Up to conjugation, the complex reflection group $G = G(m,p,2) < \GL_2$ 
has the form
\begin{equation} \label{e:gmp2-expl}
  G = \biggl\langle \begin{pmatrix} e^{2 \pi i/m} & 0 \\ 0 & e^{-2\pi i/m} \end{pmatrix}, \begin{pmatrix} e^{2 \pi p i/m} & 0 \\ 0 & 1 \end{pmatrix}, \begin{pmatrix} 0 & 1 \\ 1 & 0 \end{pmatrix} \biggr\rangle.
\end{equation}
Let $K < G$ be the index-two abelian subgroup of diagonal matrices.
As before, let $V = \CC^4$ and consider $K < G < \Sp(V)$ in the
standard way.  Let $r := m/p$.  Then, the invariants $\cO_V^K$ are
spanned by the monomials
\begin{equation}
x_1^a x_2^b y_1^c y_2^d, \quad  m \mid \bigl((a-c) - (b-d)\bigr), r \mid a,b,c,d.
\end{equation}
The invariants $\cO_V^G$ are spanned by the sums $x_1^a x_2^b y_1^c
y_2^d + x_1^b x_2^a y_1^d y_2^c$ where $a,b,c$, and $d$ are as above. It follows easily that, as an algebra, $\cO_V^G$ is generated by
\begin{gather}
x_1 y_1 + x_2 y_2, x_1 y_1 x_2 y_2, x_1^m + x_2^m, y_1^m + y_2^m, x_1^r x_2^r, y_1^r y_2^r, x_1^{jr} y_2^{m-jr} + x_2^{jr} y_1^{m-jr} (1 \leq j < p), \label{e:gm2pg} \\
x_1^{m+1} y_1 + x_2^{m+1} y_2,  y_1^{m+1} x_1 + y_2^{m+1} x_2, x_1^{jr+1} y_1 y_2^{m-jr} + x_2^{jr+1} y_2 y_1^{m-jr} (1 \leq j < p). \label{e:gm2og}
\end{gather}
The second line consists of elements obtainable from those in the
first line by a linear combination of bracketing with $x_1 y_1 x_2
y_2$ and multiplying by $x_1 y_1 + x_2 y_2$, and hence the first line
Poisson generates $\cO_V^G$.  Therefore, $\{\cO_V^G, \cO_V\}$ is spanned by $\{f, \cO_V\}$ where $f$ ranges among the elements listed in \eqref{e:gm2pg}.

In the next subsections we will consider separately the cases $p=1$,
$p=m$, and $1 < p < m$.  We first consider $p = 1$ since the computations
here will be used in subsequent subsections as well.
\begin{remark}
  The techniques used here might also be able to handle the case of
  somewhat more general finite subgroups of $\GL_2$: namely, those
  generated by a subgroup of diagonal matrices together with an
  off-diagonal element with zeros on the diagonal.  For such groups,
  we can use the subgroup $K < G$ of diagonal matrices, which has
  index two, and for which $\HP_0(\cO_V^K, \cO_V)$ was computed in the
  previous section. In more detail, there is a natural map
  $\HP_0(\cO_V^G) \into \HP_0(\cO_V^G, \cO_V) \onto \HP_0(\cO_V^K,
  \cO_V)=\HP_0(\cO_V^K)$ whose image is $\HP_0(\cO_V^K)^G$, the part
  symmetric under swapping indices $1$ and $2$. The dimension of the
  latter is roughly $\frac{1}{2} \dim \HP_0(\cO_V^K)$, so estimates
  using Theorem \ref{theorem-staircase}, in the spirit of the previous
  section, should suffice to show that $\HP_0(\cO_V^G) \ncong \gr
  \HH_0(\cD_X^G)$ for many of these $G$.
\end{remark}

\subsection{The case $p=1$}\label{s:gm12}
Set $G = G(m,1,2)$.  
\begin{theorem}\label{t:gm12}
For $G = G(m,1,2)$, $\HP_0(\cO_V^G, \cO_V) \cong \gr \HH_0(\cD_X^G, \cD_X)$, and
a homogeneous basis for the former is given by the images of the elements 
\begin{gather}\label{gmp2-gens1}
x_1^a x_2^b y_1^a y_2^b \quad (a,b \leq m-2); \quad \quad x_1^{m-1} x_2^a y_1^{m-1} y_2^a + x_1^a x_2^{m-1} y_1^a y_2^{m-1} (1 \leq a \leq m-1); \\
x_1^{a+b} y_1^a y_2^b, \quad x_2^{a+b} y_1^b y_2^a \quad (a+b \leq m-2, \, b \geq 1); \\ 
b x_1^{m-1} y_1^{m-1-b} y_2^{b} -(m-b) x_2^{m-1} y_1^{m-b} y_2^{b-1} \quad (1 \leq b \leq m-1).\label{gmp2-gens3}
\end{gather}
\end{theorem}
The $G$-graded structure of $H = \HP_0(\cO_V^G, \cO_V) \cong \gr
\HH_0(\cD_X^G, \cD_X)$ follows immediately from
\eqref{gmp2-gens1}--\eqref{gmp2-gens3}.  We will need some notation
for the irreducible representations of $G$.  Let $\chi$ be the
tautological one-dimensional representation of the group of $m$-th
roots of unity $\{e^{2\pi k i/m}\}$.  For $0 \leq a \leq m-1$, let
$\rho_a := \chi^a \circ \det$, so that $\rho_0$ is the trivial
representation. Let $\rho_0^-$ be the nontrivial one-dimensional
representation which restricts to the trivial representation on $K$,
i.e., which is $-1$ on off-diagonal elements and $1$ on diagonal
elements.  Then, let $\rho_a^- := \rho_0^- \otimes \rho_a$.  Next, for
$a \neq b$, let $\tau_{a,b}$ be the two-dimensional irreducible
representation which restricts to $(\chi^a \boxtimes \chi^b) \oplus
(\chi^b \boxtimes \chi^a)$ on $K$. There are ${m \choose 2}$ distinct such
irreducible representations. Note that the corresponding representation in the case $a=b$ is $\rho_a \oplus \rho_a^-$.
\begin{corollary} \label{c:gm12}
\begin{gather}
h(\Hom_G(\rho_0, H);t) = \sum_{j=0}^{m-2} \lfloor \frac{j+2}{2} \rfloor t^{2j} + \sum_{j=m-1}^{2m-4} \lfloor \frac{2m-2-j}{2} \rfloor t^{2j} + \sum_{j=0}^{m -2} t^{2m+2j}; \\
h(\Hom_G(\rho_0^-, H);t) = \sum_{j=0}^{m-2} \lfloor \frac{j+1}{2} \rfloor t^{2j} +
 \sum_{j=m-1}^{2m-4} \lfloor \frac{2m-3-j}{2} \rfloor t^{2j};
\end{gather}
\begin{multline}
h(\Hom_G(\tau_{b,-b}, H); t) = (t^{2b} + t^{2b+2} + \cdots + t^{2(m-b)-2})
\\ + (2t^{2(m-b)} + 2t^{2(m-b)+4} + \cdots + 2t^{2m-4}) + t^{2m-2}, \quad 1 \leq b < m/2;
\end{multline} 
If $m$ is odd, then for all other irreducible representations $\rho$, $\Hom_G(\rho, H) = 0$. If $m$ is even, then this is true except for $\rho_{m/2}$ and $\rho_{m/2}^-$, for which
\begin{gather}
 h(\Hom_G(\rho_{m/2}, H);t) = t^m + t^{m+2} + \cdots + t^{2m-4}; \\
h(\Hom_G(\rho_{m/2}^-, H);t) = t^m + t^{m+2} + \cdots + t^{2m-2}.
\end{gather}
\end{corollary}
We omit the proof of the corollary, since it follows directly from the theorem.
\begin{proof}[Proof of Theorem \ref{t:gm12}]
  We will prove that the given elements map to a basis of
  $\HP_0(\cO_V^G, \cO_V)$. From this it is easy to deduce that
  $\HP_0(\cO_V^G, \cO_V) \cong \gr \HH_0(\cD_X^G, \cD_X)$: we only
  need to compute that the dimensions are equal, since there is always
  a surjection. By Lemma \ref{l:afls-fla}, $\dim \HH_0(\cD_X^G,
  \cD_X)$ equals the number of elements $g \in G$ such that $g - \Id$
  are invertible. There are $(m-1)^2$ diagonal elements without $1$ on
  the diagonal, and $m(m-1)$ off-diagonal matrices of determinant not
  equal to $-1$, and these are exactly the elements such that $g-\Id$
  is invertible. So it is enough to show that $\dim \HP_0(\cO_V^G,
  \cO_V) = (m-1)(2m-1)$, and this follows by computing the number of
  basis elements.

  We will compute explicitly the brackets of \eqref{e:gm2pg} and show
  that the claimed elements form a basis of $\HP_0(\cO_V^G, \cO_V)$.
  Since $p=1$, only the first four elements of \eqref{e:gm2pg} are
  needed.  So, we compute the brackets with these elements.

First, $\{x_1 y_1 + x_2 y_2, \cO_V^G\}$ is the span of all monomials
$x_1^a x_2^b y_1^c y_2^d$ with $a+b \neq c+d$.

Next, $\{x_1 y_1 x_2 y_2, \cO_V^G\}$ is the span of elements $(c-a)
x_1^{a-1} x_2^b y_1^{c-1} y_2^d + (d-b) x_1^a x_2^{b-1} y_1^c
y_2^{d-1}$.  In the case $a+b=c+d$ (otherwise the monomial is in the
span of the previous paragraph), this reduces to $x_1^{a-1} x_2^b
y_1^{c-1} y_2^d + x_1^a x_2^{b-1} y_1^c y_2^{d-1}$.  So if we quotient 
 by this and the brackets of the previous paragraph, the result is spanned
by the images of the monomials 
\begin{equation}\label{e:gm2-mons}
  x_1^a x_2^b y_1^a y_2^b \quad (a,b \geq 0); \quad  \quad
  x_1^{a+b} y_1^a y_2^b, \quad x_2^{a+b} y_1^b y_2^a \quad (a \geq 0, b > 0),
\end{equation}
remembering also the equivalences
\begin{equation}\label{e:gm2-mons-equiv}
x_1^{a+b} y_1^a y_2^b \equiv x_1^{a+b-c}
x_2^c y_1^{a-c} y_2^{b+c}, \quad x_2^{a+b} y_1^b y_2^a \equiv  x_1^c x_2^{a+b-c} y_1^{b+c} y_2^{a-c}, \quad \quad c \leq a, b > 0,
\end{equation}
which we will use for subsequent relations.

Finally, $\{x_1^m + x_2^m, \cO_V^G\}$ is spanned by $c x_1^{a+m-1}
x_2^b y_1^{c-1} y_2^d +d x_1^{a} x_2^{b+m-1} y_1^c y_2^{d-1}$, and
similarly for $\{y_1^m + y_2^m, \cO_V^G\}$. In particular, this
includes the elements $x_1^{a} x_2^b y_1^{a+b}$ and $x_1^{b} x_2^{a}
y_2^{a+b}$ for $a \geq m-1$ and $b \geq 0$.  Together with the spans
described in the previous paragraphs, we can first restrict our
attention to the case $a+b+m-1=c+d-1$, i.e., $d = a+b-c+m$.  Then, we
obtain the monomials of the second two forms of \eqref{e:gm2-mons} in
the case that $a \geq m-1$, i.e.,
\begin{equation} \label{e:gm2-monrels1} x_1^{a+b} y_1^a y_2^b, \quad
  x_2^{a+b} y_1^b y_2^a \quad (a \geq m-1, b > 0).
\end{equation}
The remaining elements in the span yield, up to the symmetry of swapping $x_1$
with $x_2$ and $y_1$ with $y_2$ (and still assuming $d=a+b-c+m$),
\begin{gather}
  c x_1^{a+b+m-1} y_1^{b+c-1} y_2^{d-b} + d x_2^{a+b+m-1} y_1^{c-a} y_2^{a+d-1}, \quad \text{if $a < c, b < d$};  \notag \\
  x_1^{a+b+m-1} y_1^{b+c-1} y_2^{d-b}, \quad \text{if $a > c$}; \notag \\
(b+m) x_1^a x_2^{b+m-1} y_1^a y_2^{b+m-1} + a x_1^{a+b+m-1} y_1^{a+b-1} y_2^m, \quad \text{if $a=c$}.   \label{e:gm2-monrels2}
\end{gather}
The final expression \eqref{e:gm2-monrels2} together with
\eqref{e:gm2-monrels1} yields the first monomial of
\eqref{e:gm2-monrels2} when $a+b \geq m$, or equivalently (by changing
$a$ and $b$):
\begin{equation} \label{e:gm2-monrels3}
x_1^{a} x_2^{b} y_1^a y_2^{b}, \quad a+b \geq 2m-1.
\end{equation}
The expressions in the two lines above \eqref{e:gm2-monrels2} can be
rewritten, by changing $a,b,c,d$, as
\begin{gather} 
  c x_1^{a+b} y_1^a y_2^b +  (a+b+1-c) x_2^{a+b} y_1^{m-b} y_2^{a+2b-m} \quad (0 < m-b \leq c \leq a+1, \, b > 0); \notag \\
  x_1^{a+b} y_1^a y_2^b, \quad x_2^{a+b}y_1^by_2^a \quad (b >
  m).  \label{e:gm2-monrels4}
\end{gather}
For fixed $a$ and $b$, if there is more than one possible value of $c$
in the first equation above, then in fact both monomials that appear
are in the span. So, we can rewrite this as
\begin{gather} \label{e:gm2-monrels5} (a+1) x_1^{m-1} y_1^a
  y_2^{m-a-1} + (m-a-1) x_2^{m-1} y_1^{a+1} y_2^{m-a-2} \quad (a <
  m-1);\\ 
  \label{e:gm2-monrels6} x_1^{a+b} y_1^a y_2^b, \quad x_2^{a+b} y_1^b
  y_2^a \quad (a+b \geq m, \, 0 < b < m).
\end{gather}
Applying the aforementioned swap of indices $1$ and $2$ to \eqref{e:gm2-monrels2}, we also obtain
\begin{equation} \label{e:gm2-monrels7}
(b+m) x_1^{b+m-1} x_2^{a} y_1^{b+m-1} y_2^{a} + a x_2^{a+b+m-1} y_1^m y_2^{a+b-1}.
\end{equation}
The overall span \eqref{e:gm2-monrels1}--\eqref{e:gm2-monrels7} is now
symmetric in swapping indices $1$ and $2$. It is
also 
 almost symmetric in swapping $x$ with $y$ using 
\eqref{e:gm2-mons-equiv}, since the latter shows that
 $x_1^{a+b} y_1^a y_2^b$ is equivalent to $x_1^b x_2^a
y_2^{a+b}$ when $b > 0$.  However, \eqref{e:gm2-monrels2} yields, after
swapping $x$ with $y$ and applying \eqref{e:gm2-mons-equiv},
\[
(b+m) x_1^a x_2^{b+m-1} y_1^a y_2^{b+m-1} + a x_2^{a+b+m-1} y_1^m y_2^{a+b-1}. 
\]
Up to \eqref{e:gm2-monrels7}, this is equivalent to
\begin{equation}\label{e:gm2-monrels8}
x_1^a x_2^{b+m-1} y_1^a y_2^{b+m-1} - x_1^{b+m-1} x_2^a y_1^{b+m-1} y_2^a.
\end{equation}
We conclude that
$\HP_0(\cO_V^G, \cO_V)$ is presented as the span of monomials
\eqref{e:gm2-mons} modulo span of
\eqref{e:gm2-monrels1}--\eqref{e:gm2-monrels8}. 
From this the
statement of the theorem easily follows.
\end{proof}

\subsection{The case $p=m$, i.e., the Coxeter groups $I_2(m)$} In the
case $p=m$, $G(m,m,2)$ is the Coxeter group $I_2(m)$.  
\begin{theorem}\label{t:gmm2}
If $p=m$, then $\HP_0(\cO_V^G, \cO_V) \cong \gr \HH_0(\cO_V^G, \cO_V)$, and a homogeneous basis of the former is given by the images of the elements
\begin{equation}
x_1^a y_1^a + (-1)^a x_2^a y_2^a, 0 \leq a \leq m-2.
\end{equation}
\end{theorem}
We can immediately deduce the graded $G$-structure.  Let $\rho_0$ be the trivial representation and ``$\det$'' the determinant representation.  Let $H := \HP_0(\cO_V^G, \cO_V) \cong \gr \HH_0(\cO_V^G, \cO_V)$.
\begin{corollary}\label{c:gmm2}
\begin{equation}
h(\Hom_G(\rho_0, H);t) = 1 + t^4 + \cdots + t^{4\lfloor\frac{m-2}{2} \rfloor} \text{ and } 
h(\Hom_G(\det, H);t) = t^2 + t^6 + \cdots + t^{4\lfloor\frac{m-1}{2} \rfloor - 2}.
\end{equation}
\end{corollary}
\begin{proof}[Proof of Theorem \ref{t:gmm2}]
  As in the proof of Theorem \ref{t:gm12}, it is enough to prove that
  the claimed elements form a basis of $\HP_0(\cO_V^G, \cO_V)$, since
  there are $m-1$ basis elements and this equals the number of
  elements $g \in G$ such that $g - \Id$ is invertible (in this case,
  they are the nontrivial diagonal elements of $G$).

  To do this, we compute explicitly the remaining brackets of
  \eqref{e:gm2pg} needed.  In this case, the final element of
  \eqref{e:gm2pg} is unnecessary, since it is a scalar multiple of the
  bracket $\{x_1 x_2, y_1^m + y_2^m\}$. So, $\HP_0(\cO_V^G, \cO_V)$ is
  the quotient of the span of \eqref{e:gm2-mons} and also the
  equivalent monomials according to \eqref{e:gm2-mons-equiv}, modulo
  \eqref{e:gm2-monrels1}--\eqref{e:gm2-monrels8} together with the
  span of $\{x_1 x_2, \cO_V\} + \{y_1 y_2, \cO_V\}$.  We now compute
  these spans.

Note that 
\begin{equation}\label{e:gmmp-t1}
\{x_1 x_2, x_1^a x_2^b y_1^c y_2^d\} = c x_1^a x_2^{b+1}
y_1^{c-1} y_2^d + d x_1^{a+1} x_2^b y_1^c y_2^{d-1}.
\end{equation}
In the case
$c=0$ or $d=0$ but not both, this yields the monomial $x_1^a x_2^{b+1}
y_1^{c-1}$ or $x_1^{a+1} x_2^b y_2^{d-1}$.  Applying this to the span $\{y_1 y_2, \cO_V\}$ and changing the $a,b,c$, and $d$, we obtain the monomials
\begin{equation} \label{e:gmmp-offdiag}
x_1^c y_1^a y_2^{b},
 \quad x_2^c y_1^{b} y_2^a, \quad \quad b \geq 1.
\end{equation}
This already includes all but the first type of monomial in \eqref{e:gm2-mons}.
For the remaining type, 
let us assume $a=c-1$ and $b+1=d$ in \eqref{e:gmmp-t1}. Then we obtain the
element
\begin{equation}
(a+1) x_1^a x_2^{b+1} y_1^a y_2^{b+1} + (b+1) x_1^{a+1} x_2^b y_1^{a+1} y_2^b.
\end{equation}
By symmetry, this is the end of the new elements of $\{\cO_V^G,
\cO_V\}$ added in the case $p=m$ to those
\eqref{e:gm2-monrels1}--\eqref{e:gm2-monrels8} from the previous
section.  Note that \eqref{e:gm2-monrels2} and \eqref{e:gm2-monrels7}
together with \eqref{e:gmmp-offdiag} yields
\begin{equation}\label{e:gmmp-diag}
x_1^a x_2^{b} y_1^a y_2^{b}, \quad \text{$a \geq m-1$ or $b \geq m-1$}.
\end{equation}
Now, putting \eqref{e:gmmp-offdiag}--\eqref{e:gmmp-diag} together, applied
to the monomials
\eqref{e:gm2-mons} modulo \eqref{e:gm2-mons-equiv}, we can recover all
of the elements \eqref{e:gm2-monrels1}--\eqref{e:gm2-monrels8}, and we
easily deduce the statement of the theorem.
\end{proof}
\subsection{The case $1 < p < m$}
\begin{theorem}\label{t:gmp2}
  If $G = G(m,p,n) < \GL_2 < \Sp_4$ and $1 < p < m$, then
  a basis of $\HP_0(\cO_V^G, \cO_V)$ is obtained by the images of the elements
\begin{gather} \label{e:t-gmp21}
x_1^a x_2^b y_1^a y_2^b, \quad a < r-1, b < m-1 \text{ or } a < m-1, b < r-1; \\ \label{e:t-gmp22}
x_1^{a} x_2^{r-1} y_1^{a} y_2^{r-1} + (-1)^{a-r+1} x_1^{r-1} x_2^a y_1^{r-1} y_2^a, \quad r-1 \leq a \leq m-r-1; \\ 
\label{e:t-gmp23}
x_1 x_2^{m-1} y_1 y_2^{m-1} + x_1^{m-1} x_2 y_1^{m-1} y_2, \quad p=2; \\ 
\label{e:t-gmp24}
x_1^{a+b} y_1^a y_2^b, \quad x_2^{a+b} y_1^b y_2^a, \quad \quad b > 0,  (\text{either } a+b < 2r-1 \text{ or } a < r-1), \text{ and:} \\   
\text{$\nexists k \in [b,a+b]$ s.t.~both $\lfloor\frac{k+1}{r} \rfloor + \lfloor\frac{a+2b-k}{r} \rfloor \geq p$} \text{ and } k \neq m/2-1; \notag \\ 
x_1^{a+b} y_1^a y_2^b + x_2^{a+b} y_1^b y_2^a, \quad b > 0, (\text{either $a+b < 2r-1$ or $a < r-1$}),  \\
\notag 
a+2b \geq m,  \text{ and } \lfloor \frac{a+b+1}{r} \rfloor = p/2; \\
x_1^{a+b} y_1^a y_2^b - x_2^{a+b} y_1^b y_2^a, \quad
\frac{a+b+1}{r} = \frac{p+1}{2}, \quad \frac{b}{r} > \frac{p-1}{2}.
\end{gather}
\end{theorem}
We remark that the condition of \eqref{e:t-gmp24} in particular implies
$a+b < m-1$ (by taking $k=a+b \geq m-1$), 
so it is consistent with Theorem \ref{t:gm12}, noting that
$\HP_0(\cO_V^G, \cO_V)$ for $G=G(m,p,2)$ is a quotient of 
$\HP_0(\cO_V^H, \cO_V)$ for $H=G(m,1,2) > G$.

Also, note that the statement of the theorem actually holds when $p=m
> 2$, and reduces to Theorem \ref{t:gmm2}, but since the result is then
much simpler, we separated the two theorems.
\begin{corollary}\label{c:gmp2}
For $1 < p < m$, $\HP_0(\cO_V^G, \cO_V) \ncong \gr \HH_0(\cD_X^G, \cD_X)$.  Also, $\HP_0(\cO_V^G) \ncong \gr \HH_0(\cD_X^G)$ unless $p=2$ and $m \in \{4,6\}$, in which case one obtains
\begin{gather}\label{e:c-gmp2-m4}
h(\HP_0(\cO_V^G);t)=h(\gr \HH_0(\cD_X^G);t) = 1+t^2+3t^4+t^8, \quad G=G(4,2,2); \\
h(\HP_0(\cO_V^G);t)=h(\gr \HH_0(\cD_X^G);t) = 1+t^2+2t^4+3t^6+4t^8+t^{10}+t^{12}, \quad G=G(6,2,2). \label{e:c-gmp2-m6}
\end{gather}
In general, when $1 < p < m$,
\begin{multline}\label{e:c-gmp2}
h(\HP_0(\cO_V^G);t) = \sum_{j=0}^{2r-5} \lfloor \frac{j+2}{2} \rfloor t^{2j} + \sum_{j=2r-4}^{m-2} (r-1)t^{2j} + \sum_{j=m-1}^{m+r-4} (m+r-3-j) t^{2j}  \\
+ \sum_{j=0}^{\lfloor \frac{m-2r}{2}\rfloor} t^{4(r+j-1)}
+ \delta_{p,2} t^{2m} +  \delta_{2 \mid p} \sum_{i=0}^{r-2} t^{m+2i},
\end{multline}
where $\delta_{2 \mid p} = 1$ if $p$ is even and $\delta_{2 \mid p} = 0$ otherwise.
\end{corollary}
It is also possible to use Theorem \ref{t:gmp2} to give
an explicit description of the graded $G$-structure of $\HP_0(\cO_V^G,
\cO_V)$ similarly to Corollaries \ref{c:gm12} and \ref{c:gmm2}, but we
omit this as it is complicated and less explicit. In computing the
Hilbert series of the $G$-invariants above, the relevant basis elements
above greatly simplify.
\begin{remark}\label{r:gmp2-topdeg}
  As a consequence of the theorem, we see that, for $1 < p < m$, the
  top degree of $\HP_0(\cO_V, \cO_V^G)$ is the same as the top degree
  of $\HP_0(\cO_V^G)$, which is $2(m+r-4)$ except in the cases $p = 2$
  and $m \in \{4,6\}$ (exactly the same cases wherein $\HP_0(\cO_V^G)
  \cong \gr \HH_0(\cO_V^G)$), in which case the top degree is $2m$. In
  contrast, Theorem \ref{t:gm12} shows that, in the case $p=1$, the
  top degree is $4m-4$, which is also the same as the top degree of
  $G$-invariants; Theorem \ref{t:gmm2} shows that, in the case $p=m$
  (i.e., the Coxeter groups of type $I_2(m)$), the top degree is
  $2m-4$, while the top degree of $G$-invariants is either $2m-4$ or
  $2m-6$, whichever is a multiple of $4$. In the case $m$ is odd,
  these produce some of the only examples of groups considered in this
  paper such that the top degree of $\HP_0(\cO_V^G, \cO_V)$ exceeds
  that of $\HP_0(\cO_V^G)$: the other examples are the groups $S_{n+1}
  < \GL_n < \Sp_{2n}$ (i.e., the type $A_n$ Weyl groups). This does
  not include groups mentioned for which we did not actually compute
  $\HP_0(\cO_V^G, \cO_V)$, such as complex reflection groups of rank
  $\geq 3$ and $G_{18}$ and $G_{19}$.

  Finally, note that the actual top degrees for $G(m,p,2)$ above
  differ from the bounds of Corollary \ref{crtopdegcor} (assuming $m >
  1$): there we have $2m + 4r - 8$, whereas the actual top degree as
  above is a constant plus $2m+2r$ (the constant depending on whether
  $p=1, p=m$, or $1 < p < m$, with the special cases $(m,p) \in
  \{(4,2), (6,2)\}$).  The only cases where the bound is sharp are
  $p=m$, $(m,p) = (4,2)$, and $(m,p)=(2,2)$.
\end{remark}
\subsubsection{Proof of Theorem \ref{t:gmp2}}
We need to compute the spans of brackets with the final three
elements of \eqref{e:gm2pg}, when summed with the spans already
computed from \S \ref{s:gm12}.

First, $\{x_1^r x_2^r, x_1^a x_2^b y_1^c y_2^d\} = cr x_1^{a+r-1}
x_2^{b+r} y_1^{c-1} y_2^d + dr x_1^{a+r} x_2^{b+r-1} y_1^c y_2^{d-1}$. Together with the similar expression for brackets with $y_1^r y_2^r$, and up to
\eqref{e:gm2-mons-equiv},
 this yields the span of
\begin{gather} \label{e:gmp2-rels1}
(a+1) x_1^{a} x_2^{b+1} y_1^a y_2^{b+1} + (b+1) x_1^{a+1} x_2^b y_1^{a+1} y_2^b, \quad a, b \geq r-1; \\ \label{e:gmp2-rels2}
x_1^{a+b} y_1^a y_2^b, x_2^{a+b} y_1^b y_2^a, \quad a+b \geq 2r-1, a \geq r-1, b > 0.
\end{gather}
Together with \eqref{e:gm2-monrels2}, since $m \geq 2r$, this also yields
\begin{equation} \label{e:gmp2-rels3}
x_1^a x_2^b y_1^a y_2^b, \quad a+b \geq r+m-1.
\end{equation}

It remains to consider the final element of \eqref{e:gm2pg} (note that $m-jr=(p-j)r$):
\begin{multline} \label{e:gmp2-t1}
\{x_1^{jr} y_2^{m-jr} + x_2^{jr} y_1^{m-jr}, x_1^a x_2^b y_1^c y_2^d\}
= r\Bigl[
\bigl(jc x_1^{a+jr-1} x_2^{b} y_1^{c-1} y_2^{d+m-jr} - (p-j)b x_1^{a+jr}
x_2^{b-1} y_1^c y_2^{d+m-jr-1}\bigr) \\
- \bigl((p-j)a x_1^{a-1} x_2^{b+jr} y_1^{c+m-jr-1} y_2^d -
jd x_1^{a} x_2^{b+jr-1} y_1^{c+m-jr} y_2^{d-1}\bigr)
\Bigr] 
\end{multline}
We will assume that $(a+jr-1)+b = c-1+(d+m-jr)$, since otherwise the
above is all in the span of $\{x_1 y_1 + x_2 y_2, \cO_V^G\}$ as noted
in \S \ref{s:gm12}.

In the case $a+jr=c$, so that the first two terms on the RHS have the
form $x_1^{a'} x_2^{b'} y_1^{a'} y_2^{b'}$, we can simplify the above
using \eqref{e:gmp2-rels1}.  We can restrict our attention to the case
that $a+d < r$, since otherwise all the terms on the RHS are already
in the span of \eqref{e:gmp2-rels2} and \eqref{e:gmp2-rels3}, using
also the relations \eqref{e:gm2-mons-equiv}.  Then,
up to the previous spans and rescaling we obtain
\begin{equation}\label{e:gmp2-t4}
p(a+jr) x_1^{a+jr-1} x_2^{d+m-jr} y_1^{a+jr-1} y_2^{d+m-jr}
- ((p-j)a-jd) x_2^{a+d+m-1} y_1^{m} y_2^{a+d-1}.
\end{equation}
In the case $a=d=0$, the second term vanishes and we obtain the monomial
$x_1^{jr-1} x_2^{m-jr} y_1^{jr-1} y_2^{m-jr}$ in the span. Otherwise,
substituting \eqref{e:gm2-monrels7}, this is equivalent to
\begin{equation}\label{e:gmp2-t2}
(a+d)p(a+jr) x_1^{a+jr-1} x_2^{d+m-jr} y_1^{a+jr-1} y_2^{d+m-jr}
+ m((p-j)a-jd) x_1^{a+d} x_2^{m-1} y_1^{a+d} y_2^{m-1}.
\end{equation}

If, instead of $a+jr=c$, we have $b+jr = d$,
i.e., the second two terms on the RHS of \eqref{e:gmp2-t1} have the
form $x_1^{a'} x_2^{b'} y_1^{a'} y_2^{b'}$ (rather than the first two
terms), then up to \eqref{e:gm2-monrels8} and swapping $j$
with $p-j$, we obtain the same relations.

Let analyze \eqref{e:gmp2-t4} and \eqref{e:gmp2-t2} further.
Using \eqref{e:gmp2-t2} together with \eqref{e:gmp2-rels1} (and the case
$a=d=0$ of \eqref{e:gmp2-t4}), we can
replace all monomials of the form $x_1^a x_2^b y_1^a y_2^b$ for $a,b
\geq r-1$ and $a+b \geq m-1$ by monomials of the form $x_1^{a+b-m+1}
x_2^{m-1} y_1^{a+b-m+1} y_2^{m-1}$ as above. It remains to see when
two such ways, for fixed $a$ and $b$, are irredundant, and hence
$x_1^{a+b-m+1} x_2^{m-1} y_1^{a+b-m+1} y_2^{m-1}$ is itself in the
span. We already saw that the latter is true when $a+b=m-1$, by \eqref{e:gm2-monrels2}.

In the case that $a=0$ and $d=1$ of \eqref{e:gmp2-t4}, then \eqref{e:gmp2-t2}
becomes, after dividing by $mj$,
\begin{equation}\label{e:gmp2-a0d1}
x_1^{jr-1} x_2^{m-jr+1} y_1^{jr-1} y_2^{m-jr+1} - x_1 x_2^{m-1} y_1 x_2^{m-1}.
\end{equation}
In the case that $a=1$ and $d=0$ of \eqref{e:gmp2-t4}, applying also
\eqref{e:gmp2-rels1}, we obtain
\[
-\frac{jr}{m-jr+1} p(1+jr)x_1^{jr-1} x_2^{m-jr+1} y_1^{jr-1} y_2^{m-jr+1} + m(p-j) x_1 x_2^{m-1} y_1 y_2^{m-1}.
\]
Together with \eqref{e:gmp2-a0d1}, this yields both monomials above, and in
particular $x_1 x_2^{m-1} y_1 x_2^{m-1}$, unless $jrp(1+jr) = m(p-j)(m-jr+1)$. 
Substituting $m=pr$ this equality becomes
$\frac{j}{p-j} = \frac{(p-j)r+1}{jr+1}$.  This holds if and only if $j=p-j$: if $j \neq p-j$, then one is strictly between both sides.   Note further that, unless $p=2$,
then we can always choose $j$ so that $j \neq p-j$, and therefore we obtain the monomial $x_1 x_2^{m-1} y_1 y_2^{m-1}$ in the span.

In the case that $a+d > 1$, then \eqref{e:gmp2-t2} can be applied to
at least three pairs $(a,d)$ with the same sum, and it is easy to see
that the second monomial (which does not change) must be in the span,
and hence all the monomials which appear are in the span. To summarize, \eqref{e:gmp2-t1} yields, in the case $c=a+jr$,
\begin{equation}\label{e:gmp2-rels4}
x_1^a x_2^{m-1} y_1^a y_2^{m-1}, a \geq 2 \text{ or } a=1, p > 2.
\end{equation}

In the remaining case of \eqref{e:gmp2-t1} where neither $a+jr=c$ nor
$b+jr=d$, provided $c, d \geq 1$, using \eqref{e:gm2-mons-equiv},
\eqref{e:gmp2-t1} becomes
\begin{equation} \label{e:gmp2-t3}
(jc-(p-j)b)x_1^{a+jr-1} x_2^{b} y_1^{c-1} y_2^{d+m-jr} - ((p-j)a-jd)x_1^{a} x_2^{b+jr-1} y_1^{c+m-jr} y_2^{d-1}.
\end{equation}
As before, we assume that the total degree in $x_1$ and $x_2$ equals the total
degree in $y_1$ and $y_2$, i.e., $a+b+jr = c+d+m-jr$. In particular, $a+b \equiv c+d \pmod r$.

If $c = 0$ and/or $d = 0$, then we instead get the same relation as
above, except that we must multiply the first term above by $\frac{x_1
  y_1}{x_2 y_2}$ and/or the second term by $\frac{x_2 y_2}{x_1 y_1}$,
respectively. (Note that if $b=c=0$ then the first term is zero, and
if $a=d=0$ then the second term is zero.)

The first term above vanishes if and only if $jc=(p-j)b$, and the second term
if and only if $jd=(p-j)a$.  One way the first equality can hold is if
$b=c=0$, in which case the second monomial appearing above is in the overall
span unless $jd=(p-j)a$, in which case we obtain no relations.  If $b+c = 1$,
then the first term does not vanish, and we obtain a nontrivial relation.
If $b+c > 1$ and
either $a,c > 0$ or $b, d > 0$, then we can replace $(a,b,c,d)$ by $(a
\pm 1, b \mp 1, c \pm 1, d \mp 1)$, and together with
\eqref{e:gm2-mons-equiv}, the new expression \eqref{e:gmp2-t3} is
irredundant unless $j=p-j$. 

The same arguments apply if we swap $b$ and $c$ with $a$ and $d$.  So, all
the monomials that can occur above are actually in the span, unless we
are in one of the cases $b+c=1=a+d$, one of $a,c$ and one of $b,d$ are
zero, or $j=p-j$ and $b+c, a+d > 0$. Even if we are in one of these
cases, by applying also \eqref{e:gm2-mons-equiv}, we can still obtain
the first monomial in the span if $b+c \geq r$, and the second
monomial in the span if $a+d \geq r$.  We can therefore discard the
case $b+c=1=a+d$, since this together with $b+c < r$ and $a+d < r$
already implies $j=p-j$.

Next, let us assume that $b+c < r$ and $a+d < r$, in addition to being
in one of the two cases (i) one of $a,c$ and one of $b,d$ are zero, or
(ii) $j=p-j$ and $b+c, a+d > 0$.  Then, applying again
\eqref{e:gm2-mons-equiv}, we obtain a single nontrivial relation
unless either $a=d=0$ and $jc=(p-j)b$ are both satisfied or $b=c=0$
and $jd=(p-j)a$ are both satisfied.  Then, we are in case (i), so $j =
p-j$, and either (1) both $a=d=0$ and $b=c < r$ are satisfied, or (2)
both $a=d < r$ and $b=c=0$ are satisfied.  So, in these final two
subcases (1) and (2) only, \eqref{e:gmp2-t3} yields no relations on
the monomials \eqref{e:gm2-mons} modulo \eqref{e:gm2-mons-equiv}, and
otherwise we obtain a single nontrivial relation.

Putting everything together, one may verify that
\eqref{e:gmp2-t3} adds to the overall span of $\{\cO_V^G, \cO_V\}$ exactly the
following:
\begin{gather} \label{e:gmp2-rels5beg}
x_1^{a+b} y_1^a y_2^b, \quad b > 0,\, \lfloor \frac{a+b+1}{r} \rfloor \neq p/2, \frac{a+b+1}{r} \neq \frac{p+1}{2}, \text{ and } \\ \notag
\exists k \in [b, a+b] \text{ s.t. } \lfloor \frac{k+1}{r} \rfloor
+ \lfloor \frac{a+2b-k}{r} \rfloor \geq p; \\
x_1^{a+b} y_1^a y_2^b - x_2^{a+b} y_1^b y_2^a, \quad b>0, \lfloor \frac{a+b+1}{r} \rfloor = p/2, a+2b \geq m; \\
x_1^{a+b} y_1^a y_2^b + x_2^{a+b} y_1^b y_2^a, \quad \frac{a+b+1}{r} = \frac{p+1}{2}, \quad \frac{b}{r} >  \frac{p-1}{2}. \label{e:gmp2-rels5end}
\end{gather}
Therefore, $\HP_0(\cO_V^G, \cO_V)$ is the quotient of the span of
monomials \eqref{e:gm2-mons} up to \eqref{e:gm2-mons-equiv} and the
relations \eqref{e:gmp2-rels1}--\eqref{e:gmp2-rels3},
\eqref{e:gmp2-rels4}, and
\eqref{e:gmp2-rels5beg}--\eqref{e:gmp2-rels5end}.  From this,
we easily obtain the basis claimed in the theorem. (A priori, we might also
need to include relations from \S \ref{s:gm12}, but it is easy to see they are
all spanned by the present relations, by
comparing the basis of the present theorem with
that of
Theorem
\ref{t:gm12}.  Alternatively, one can verify directly that the aforementioned
relations span also \eqref{e:gm2-monrels1}--\eqref{e:gm2-monrels8}.
This completes the proof of Theorem \ref{t:gmp2}.

\subsubsection{Proof of Corollary \ref{c:gmp2}}
First, to prove \eqref{e:c-gmp2}, we can use the basis of the theorem:
it is easy to see that the dimension of the space of $G$-invariants in
each degree is the number of terms of the form $x_1^a x_2^b y_1^a
y_2^b + x_1^b x_2^a y_1^b y_2^a$ and, in the case $p$ is even, also 
$x_1^{a+m/2} y_1^a y_2^{m/2} +
x_2^{a+m/2} y_1^{m/2} y_1^a$, which are in the span of the elements
appearing in the theorem.
From this \eqref{e:c-gmp2} easily follows.

Now, \eqref{e:c-gmp2} implies that the LHS and RHS of 
\eqref{e:c-gmp2-m4} are equal by substituting in the given values of
$m$ and $p$, and similarly for \eqref{e:c-gmp2-m6}.  To deduce from
this that $\HP_0(\cO_V^G) \cong \gr \HH_0(\cD_X^G)$ in the cases $p=2$
and $m \in \{4,6\}$, and hence the equality with the second term in
the these two equations, it suffices to show that $\dim \HP_0(\cO_V^G)
= \dim \HH_0(\cD_X^G)$.  By Lemma \ref{l:afls-fla}, $\dim
\HH_0(\cD_X^G, \cD_X)$ and $\dim \HH_0(\cD_X^G)$ equal the number of
elements $g \in G$ such that $g - \Id$ is invertible and the number of
conjugacy classes of such elements, respectively.  First, there are
$(m-r)r + (r-1)^2$ diagonal matrices in $G$ without $1$ on the
diagonal; of these there are $r-1$ or $2r-1$ scalar matrices,
depending on whether $p$ is odd or even, respectively.  The diagonal
matrices with distinct diagonal entries appear in conjugacy classes of
size two.  Next, the off-diagonal matrices $g$ such that $g-\Id$ is
invertible are those of determinant not equal to $-1$, i.e., equal to
a nontrivial $r$-th root of unity. There are $m(r-1)$ of these. Their
conjugacy classes are of size either $m$ (in the case $p$ is odd) or
$m/2$ (in the case $p$ is even).  Putting this together, we conclude
\begin{gather} \label{e:c-gmp2-hh0dim}
\dim\HH_0(\cD_X^G, \cD_X) = (2r-1)(m-1), \\ \label{e:c-gmp2-hh0invdim}
\dim \HH_0(\cD_X^G) = \begin{cases} \frac{1}{2}r(m+1) - 1, & \text{$p$ is odd}, \\ 
\frac{1}{2}r(m+4)  - 2, & \text{$p$ is even}. \end{cases}
\end{gather}
We easily deduce from this and \eqref{e:c-gmp2-m4} and
\eqref{e:c-gmp2-m6} the fact that $\dim \HH_0(\cD_X^G) = \dim
\HP_0(\cO_V^G)$ in these cases. Moreover, using
\eqref{e:c-gmp2-hh0dim} and an explicit calculation from the basis
given in the theorem, or using computer programs from Magma, we see
that $\dim \HP_0(\cO_V^G, \cO_V) > \dim \HH_0(\cD_X^G, \cD_X)$ in
these cases: for $(m,p) = (4,2)$, we obtain dimensions $12 > 9$, and in
the case $(m,p)=(6,2)$, we obtain dimensions $34 > 25$.

It remains to prove that, in all other cases (i.e., other than $p=2$
and $m \in \{4,6\}$) $1 < p < m$ implies that $\HP_0(\cO_V^G) \not
\cong \gr \HH_0(\cD_X^G)$, since this clearly implies $\HP_0(\cO_V^G,
\cO_V) \ncong \gr \HH_0(\cD_X^G, \cD_X)$.  For this, it suffices
to show that $\dim \HP_0(\cO_V^G) > \dim \HH_0(\cD_X^G)$.  From \eqref{e:c-gmp2} we can easily compute $\dim \HP_0(\cO_V^G)$ by plugging in $t=1$; or we
can compute it from the theorem itself and the
observations of the first paragraph of the proof. The
first line becomes the number of elements of the form $x_1^a x_2^b y_1^a y_2^b + x_1^b x_2^a y_1^b y_2^a$ with $a \leq b \leq m-2$ and $a \leq r-2$, which
is the area of an obvious trapezoid in the plane: $(r-1)(m-1) - \frac{1}{2}(r-1)(r-2)$. The evaluation of the second line of \eqref{e:c-gmp2} at $t=1$
is $\delta_{p,2} + \lfloor \frac{m-2r}{2} \rfloor + 1 + (r-1)\delta_{2 \mid p}$.
Put together, 
\begin{equation}\label{e:c-gmp2-hp0invdim}
\dim \HP_0(\cO_V^G) =  (r-1)(m-\frac{1}{2}r) + \lfloor \frac{m-2r}{2} \rfloor + (r-1)\delta_{2 \mid p} + 1 + \delta_{p,2}.
\end{equation}
Since the value of the formula in \eqref{e:c-gmp2-hh0invdim} for the
even case of $p$ exceeds that of the odd case, let us subtract the
even case formula from \eqref{e:c-gmp2-hp0invdim} and try to see when
the result is positive. We get:
\begin{equation}\label{e:c-gmp2-pf-t1}
(\frac{1}{2} r-1)(m-r-5) 
+ \lfloor \frac{(p-2)r}{2} \rfloor + (r-1)\delta_{2 \mid p} -2 + \delta_{p,2}.
\end{equation}
All of the terms above except for the first sum to a nonnegative
number
unless $p = 3$ and $r=2$.   The first term will be positive whenever $r > 2$ and $(p-1)r > 5$; the second condition is satisfied
for all pairs $(p,r)$ with $r > 2$ except when $p=2$ and $r \in \{3,4,5\}$.  It remains to check these last cases (along with $r=2$).

If $r=2$, then the above sum is positive unless either $p=2$ or $p=3$.
If $p=2$ and $r \in \{3,4,5\}$, then the above is clearly positive
unless $r=3$.  So this leaves only the cases $(p,r) \in \{(2,2),
(2,3), (3,2)\}$.  The first two cases are those in which the above is
zero and we actually get $\HP_0(\cO_V^G) \cong \gr \HH_0(\cD_X^G)$.
In the final case $p=3, r = 2$, $\dim \HP_0(\cO_V^G) = 7 > 6 = \dim
\HH_0(\cD_X^G)$ (recall that \eqref{e:c-gmp2-pf-t1} used the formula
\eqref{e:c-gmp2-hh0invdim} in the case $p$ is even). This completes the proof.

\appendix
\section{Examples where $\HP_0(\cO_V^G)$ is nontrivial in cubic
  degree} \label{s:nontriv-cubic} Let $G$ be a group and $V_1, V_2$,
and $V_3$ three quaternionic irreducible representations: then
$(\Sym^2 V_i)^G = 0$ for all $i \in \{1,2,3\}$. If, furthermore, $(V_i
\otimes V_j)^G = 0$ for all $i \neq j$ and $(V_1 \otimes V_2 \otimes
V_3)^G \neq 0$, then it would follow that the lowest degree invariant
element in $\cO_{V_1 \oplus V_2 \oplus V_3}^G$ is cubic.  Equipping $V
:= V_1 \oplus V_2 \oplus V_3$ with a $G$-invariant symplectic form,
$\HP_0(\cO_V^G)$ would have a nontrivial cubic component, isomorphic
to the cubic part of $\cO_{V}^G$ itself.  Our goal is to construct
such $G, V_1, V_2$, and $V_3$.

  To do so, we will employ the field $\FF_2$
  and the Arf invariant. Let $m \geq 1$ and let $E$ be a
  $\FF_2$-vector space of dimension $2m$. Let $Q_E$ denote the group
  of quadratic forms on $E$ with values in $\FF_2$.  Corresponding to
  each $q \in Q_E$ is a canonical central extension $\widetilde{E}_q$
  of $E$ by $\FF_2$, since $H^2(E, \FF_2) = Q_E$. If $q$ is
  nondegenerate, then it is well known \cite{Diclg,Arfuqf} that $(E,q)$
  is isomorphic to either $U_0^{m}$ or $U_0^{m-1} \oplus U_1$, where
  $U_0$ and $U_1$ are defined as $\FF_2^{2}$ with the quadratic forms
  $x_1 x_2$ and $x_1^2 + x_1 x_2 + x_2^2$, respectively.  In the
  former case, $q$ is said to have Arf invariant $0$, and in the
  latter case, Arf invariant $1$; the Arf invariant is the value that
  $q$ attains on the majority of vectors.

  It follows that, if $q$ is nondegenerate, then $\widetilde{E}_q$ has
  a (unique) irreducible representation $Y_q$ of dimension $2^m$ (note
  that any such irreducible representation must be unique and of
  maximal dimension, since $|\widetilde{E}_q| = 2^{2m+1}$ equals the
  sum of squares of dimensions of the irreducible representations).
  Namely, if $q = q_1 \oplus \cdots \oplus q_m$, then
  $\widetilde{E}_{q_1 \oplus \cdots \oplus q_m}$ is a central quotient
  of $\prod_i \widetilde{E}_{q_i}$, and $Y_q = Y_{q_1} \boxtimes \cdots
  \boxtimes Y_{q_m}$. This reduces one to the case $m=1$, where the
  central extensions corresponding to $U_0$ and $U_1$ are just the
  dihedral and quaternion groups of order eight, each equipped with a
  (unique) irreducible $2$-dimensional representation.  It also
  follows that $Y_q$ is equipped with a canonical
  $\widetilde{E}_q$-invariant bilinear form, which is symmetric or
  skew-symmetric, depending on whether the Arf invariant of $q$ is $0$
  or $1$, respectively (since this is true in the case $m=1$). That
  is, $Y_q$ is real or quaternionic, respectively.

Next, there is a canonical 
  group which puts together all the central extensions for
  varying $q$: Let $Q_E$ be the $\FF_2$-vector space of quadratic
  forms on $E$. Then $H^2(E,\FF_2) = Q_E$ and so there is a canonical
  element of $H^2(E, Q_E^*)$ yielding a central extension
\[
1 \to Q_E^* \to G \to E \to 1.
\]
Then, $G$ also acts on $Y_q$ with action factoring through $\widetilde
E_q$, which is the pushout of the above extension under the evaluation
map $q: Q_E^* \to \FF_2$.  It follows that $Y_q$ is an irreducible
representation of $G$ that is real or quaternionic, depending on
whether the Arf invariant of $q$ is $0$ or $1$, respectively.
Moreover, for distinct nondegenerate quadratic forms $q_1, q_2$,
$Y_{q_1} \ncong Y_{q_2}$.  Furthermore, one may check that, if
$q_1+q_2$ is nondegenerate, then $Y_{q_1} \otimes Y_{q_2} \cong Y_{q_1
  + q_2}^{2^m}$.

Now, suppose that we are given quadratic forms $q_1$ and $q_2$ of Arf
invariant $1$ such that $q_1 + q_2$ is nondegenerate and also has Arf
invariant $1$.  Then, setting $q_3 := q_1 + q_2$, we deduce that
$(\Sym^2 Y_{q_i})^G = 0$ and $(Y_{q_i} \otimes Y_{q_j})^G = 0$ for all $i
\neq j$, but since $q_1+q_2 = q_3$, $(Y_{q_1} \otimes Y_{q_2}
\otimes Y_{q_3})^G \neq 0$.  Thus, $G, Y_{q_1}, Y_{q_2}$, and
$Y_{q_1+q_2}$ provide an example of the desired form. In fact, in this
case, setting $V := Y_{q_1} \oplus Y_{q_2} \oplus Y_{q_3}$,
the cubic part of $\Sym(\cO_V^G)$
and hence $\HP_0(\cO_V^G)$ is isomorphic to $(Y_{q_1} \otimes Y_{q_2}
\otimes Y_{q_3})^G$, which is $2^m$-dimensional.

It is not hard to find such examples. Using Magma we found several
with $m=2$ (the minimum possible value), such as 
$q_1 = x_1x_2 + x_3^2+x_3x_4+x_4^2$ and
$q_2 = x_1^2+x_1x_4+x_2^2+x_2x_3+x_3x_4$.
 In this
case, setting $V := Y_{q_1} \oplus Y_{q_2} \oplus Y_{q_1+q_2}$, the space
$\HP_0(\cO_V^G)$ is nonzero in cubic degree (where it has dimension
four), and $\dim V = 12$.

\bibliographystyle{amsalpha}
\bibliography{master}

\end{document}